\documentclass[10pt,reqno,oneside]{amsart}

\usepackage[a4paper,margin=1.1in]{geometry}
\usepackage{amsmath,amssymb,amsthm,mathtools,mathrsfs}
\usepackage{thmtools}
\usepackage{enumitem}
\usepackage{hyperref}
\usepackage[nameinlink]{cleveref}
\AddToHook{cmd/appendix/before}{%
  \crefalias{section}{appendix}%
}
\usepackage{microtype}
\usepackage[dvipsnames]{xcolor}
\usepackage{comment}

\makeatletter
\renewcommand{\subsection}{\@startsection{subsection}{2}%
  \z@{\linespacing\@plus.7\linespacing}{.5\linespacing}%
  {\normalfont\scshape}}
\makeatother

\hypersetup{colorlinks=true,linkcolor=blue,citecolor=blue,urlcolor=blue}

\newcommand{\R}{\mathbb{R}}

\newcommand{\N}{\mathbb{N}}

\newcommand{\dd}{\,\mathrm{d}}
\newcommand{\p}{\partial}
\newcommand{\supp}{\operatorname{supp}}
\newcommand{\abs}[1]{\left|#1\right|}
\newcommand{\norm}[1]{\left\|#1\right\|}
\newcommand{\one}{\mathbf{1}}
\newcommand{\loc}{\operatorname{loc}}
\newcommand{\eps}{\varepsilon}

\newtheorem{theorem}{Theorem}
\newtheorem{proposition}[theorem]{Proposition}
\newtheorem{lemma}[theorem]{Lemma}
\newtheorem{corollary}[theorem]{Corollary}
\newtheorem{remark}[theorem]{Remark}
\newtheorem{definition}[theorem]{Definition}

\title[]{Blow-up for a Semilinear Tricomi-type Equation with Scale-Invariant Mass in the Oscillatory Regime}

\author{Diego Marcon, Wanderley Nunes do Nascimento, Matheus C. Santos}

\thanks{%
\begin{tabular}{@{}l}
Instituto de Matemática e Estatística, Universidade Federal do Rio Grande do Sul, Porto Alegre, Brazil. \\
\end{tabular}}

\date{}

\begin{document}

\begin{abstract}
We investigate the finite-time blow-up of solutions to a Tricomi-type equation with scale-invariant potential and power nonlinearities in the oscillatory regime. For smooth, compactly supported, nonnegative initial data, we prove nonexistence of global-in-time solutions when the power nonlinearity lies below the positive root of an explicit Strauss-type polynomial naturally associated with the equation. The proof combines two main ingredients. The first is the construction of a positive adjoint temporal profile, which yields a weighted monotonicity formula and, consequently, a quantitative lower bound for the nonlinear term. The second is a phase-localized test function argument on logarithmic time shells, fitted to capture the oscillatory effects induced by the scale-invariant potential and to derive a complementary upper bound for the same quantity. The existence of global solutions when the power nonlinearity is equal to the polynomial root is still an open problem.

\bigskip

\noindent{\it 2020 Mathematics Subject Classification:} 35B44, 35Q05, 35L15, 35L71

\noindent\textit{Keywords:} Blow-up, Generalized Tricomi equation, Scale-invariant mass, Oscillatory regime.
\end{abstract}

\maketitle

\section{Introduction}

In this paper, we study finite-time blow-up for weak solutions of the Cauchy
problem
\begin{equation}\label{eq:intro-main}
\begin{cases}
u_{tt}-(1+t)^{2\ell}\Delta u+\dfrac{\mu^2}{(1+t)^2}u=|u|^p,
& (x,t)\in\mathbb R^n\times(0,T),\\[1em]
u(x,0)=\varepsilon f(x),\qquad
u_t(x,0)=\varepsilon g(x),
& x\in\mathbb R^n,
\end{cases}
\end{equation}
where $n\ge1$,$-1<\ell<0$, $\mu^2>\frac14$, $p>1$, and $\varepsilon>0$.

The initial data are assumed to be nonnegative, compactly supported, and not
both identically zero.  The coefficient $(1+t)^{2\ell}$ describes a
decreasing propagation speed which leads the differential operator to have the shifted Tricomi structure $\partial_t^2-(1+t)^{2\ell}\Delta$, while the term $\mu^2(1+t)^{-2}u$ is a
scale-invariant time-dependent mass. In this sense, \eqref{eq:intro-main} is a shifted
Tricomi-type wave equation with scale-invariant mass.

The results of this work are motivated by the theory for the classical semilinear wave equation
\begin{equation}\label{eq:intro-classical-wave}
v_{tt}-\Delta v=|v|^p,
\qquad (x,t)\in\mathbb R^n\times(0,T).
\end{equation}
For $n\ge 2$, this model exhibits a critical threshold $p_*=p_{\mathrm{Str}}(n)$,
now known as the Strauss exponent. It is defined as the positive root of
\begin{equation}\label{eq:strauss-polinomio}
(n-1)p^2-(n+1)p-2=0.    
\end{equation}
The Strauss exponent separates the small-data dynamics into finite-time
blow-up in the subcritical and critical range $1<p\le p_{\mathrm{Str}}(n)$, and
global existence in the supercritical range $p>p_{\mathrm{Str}}(n)$.

The blow-up side, together with the identification of the expected
threshold, goes back to classical works \cite{John1979,Schaeffer1985,Sideris1984,Strauss1981}. The corresponding global existence theory in the supercritical range, as well as the resolution of the critical case, was completed through a long series
of contributions, including
\cite{Glassey,Kubo,LindbladSogge1996,TakamuraWakasa2011,YZ2006,Zhou95,Zhou2007,ZhouHan2014}.

A large part of the modern theory concerns variations of
\eqref{eq:intro-classical-wave} by scale-invariant lower-order terms such as
\begin{equation}\label{eq:intro-general-damping-mass}
w_{tt}-(1+t)^{2\ell}\Delta w
+\dfrac{\beta}{1+t}\, w_t
+\dfrac{\mu^2}{(1+t)^2}w
=|w|^p,
 \quad (x,t)\in\mathbb R^n\times(0,T).
\end{equation}
This equation belongs to the Euler-Poisson-Darboux-Tricomi or generalized Tricomi-type family.  The equation
\eqref{eq:intro-general-damping-mass}, and closely related models, has been
studied extensively from the viewpoints of global existence, blow-up and
lifespan estimates; see, among others,
\cite{HHHT22,DLR2015,FanLaiTakamura2026,LaiPalmieriTakamura2026,LaiTakamuraWakasa2017,TuLin2019Damping,NPR2017,Palmieri2025,
PR2019,PalTu20219,Wakasugi2014}.
These works show that the linear behavior of \eqref{eq:intro-general-damping-mass} is governed by
the indicial discriminant
\[
\delta:=(1-\beta)^2-4\mu^2.
\]
Indeed, the spatially homogeneous part of the corresponding linear equation
has the following Cauchy--Euler structure
\begin{eqnarray} \label{CauchyEulerODE}
h''(t)+\frac{\beta}{1+t}h'(t)
+\frac{\mu^2}{(1+t)^2}h(t)=0.
\end{eqnarray}
We look for solutions of the form $h(t)=(1+t)^r$, and obtain the indicial
equation
\[
r^2+(\beta-1)r+\mu^2=0,
\]
whose discriminant is precisely $\delta = (1-\beta)^2-4\mu^2$.  Thus,
$\delta\ge0$ corresponds to real indicial roots which imply
power-like profiles, whereas $\delta<0$ corresponds to complex indicial
roots and, therefore, logarithmically oscillatory profiles.

In the regimes treated in the existing scale-invariant damping and mass
literature, most available results treat the nonoscillatory range $\delta\ge0$. 
For the model \eqref{eq:intro-general-damping-mass}, we also observe a
critical-exponent phenomenon, whose value depends
on the strength of the scale-invariant lower-order terms.  In the purely wave
case with scale-invariant damping, critical exponents, global existence,
blow-up, and lifespan estimates are studied in
\cite{DLR2015,LaiTakamuraWakasa2017,Wakasugi2014}. The additional presence
of a scale-invariant mass is treated in
\cite{TuLin2019Damping,NPR2017,PR2019,PalTu20219}, where the interaction
between damping and mass is shown to affect both the effective dimension and
the relevant blow-up threshold.  For
Euler-Poisson-Darboux-Tricomi analogues, related critical and subcritical
problems are investigated in
\cite{HHHT22,ChLuPa2021,FanLaiTakamura2026,HeWiY20172,HeWiY2017,HeWiY2020,HeWiY2021,LaiPalmieriTakamura2026,LinTu2019Tricomi,Palmieri2025}.

Still in the nonoscillatory range, this critical behavior is not
governed by a single mechanism.  Rather, there is a competition between a
heat-like phenomenon and a wave-like phenomenon.  The heat-like side refers
to the diffusive behavior familiar from the semilinear heat equation, for
which Fujita identified the critical power $p_F(n)=1+2/n$; see
\cite{Fujita1966}. 
In damped wave models, sufficiently effective damping drives the
large-time behavior toward a diffusive, parabolic regime. Consequently,
the critical threshold may be governed by a Fujita-type exponent
associated with the parabolic asymptotic profile. On the other hand,
when the damping is not strong enough to suppress the hyperbolic character
of the equation, the wave-like mechanism dominates and the relevant threshold
is of Strauss type.  This competition is made explicit in the
scale-invariant damping and mass setting in \cite{PR2019} and is consistent
with the blow-up and lifespan estimates obtained in
\cite{DLR2015,LaiTakamuraWakasa2017,TuLin2019Damping,NPR2017,Palmieri2025,PalTu20219,Wakasugi2014}. 

A common tool in the blow-up analysis of these wave-like regimes is the use
of a solution of the adjoint homogeneous equation.  For
\eqref{eq:intro-general-damping-mass}, the formal adjoint equation is
\[
\psi_{tt}-(1+t)^{2\ell}\Delta\psi
-\left(\frac{\beta}{1+t}\psi\right)_{\!t}
+\frac{\mu^2}{(1+t)^2}\psi=0.
\]
If we look for a solution with separated variables of the form
\[
\psi(t,x)=m_\lambda(t)\varphi_\lambda(x),
\]
where $\varphi_\lambda$ is the spherical eigenfunction of the Laplacian, see \cite{YZ2006}; then, the temporal factor $m_{\lambda}$  solves
\begin{equation}\label{eq:intro-adjoint-ode}
m_\lambda''(t)
-\frac{\beta}{1+t}m_\lambda'(t)
+\left(
\frac{\beta+\mu^2}{(1+t)^2}
-\lambda^2(1+t)^{2\ell}
\right)m_\lambda(t)=0.
\end{equation}
This construction,
or close variants, lies behind many test-function proofs of blow-up.
In the classical wave equation, the mode
$e^{-t}\varphi_1(x)$ is used in
\cite{YZ2006} to prove critical blow-up.  Related adjoint-mode or positive-test-function
arguments appear in 
\cite{FanLaiTakamura2026,LaiPalmieriTakamura2026,LaiTakamuraWakasa2017,TuLin2019Damping,Palmieri2025,PalTu20219,TakamuraWakasa2011,WakaYor19,Zhou2007,ZhouHan2014}.

The present paper is devoted to the complementary oscillatory regime
$\delta<0$ which is not covered by most of the existing literature.  In view of \Cref{cor:EPDT-equation}, it suffices to consider
the reduced case $\beta=0$.  With this normalization, we have
\[
\delta<0
\quad\Longleftrightarrow\quad
\mu^2>\frac14.
\]
In this case, the associated Cauchy-Euler equation \eqref{CauchyEulerODE}, with $\beta =0$, has solutions of the form
\[
h(t)
=
(1+t)^{1/2}
\Big[
c_1\cos\!\left(\omega\log(1+t)\right)
+
c_2\sin\!\left(\omega\log(1+t)\right)
\Big],
\quad \text{with } \quad 
\omega:=\sqrt{\mu^2-\frac14} > 0.
\]
Thus, the homogeneous dynamics changes sign on logarithmic time scales and
cannot be used globally as positive test functions.  This
loss of positivity prevents the direct use of the usual space-average and
integral-functional methods \cite{Sideris1984}, the Kato-type comparison
argument \cite{Kato1980,LaScTa2020}, and the positive test-function methods
based on monotone time cutoffs or slicing procedures
\cite{TakamuraWakasa2011,WakaYor19,YZ2006,Zhou2007,ZhouHan2014}.

Our blow-up range is described by the Strauss-type polynomial
\begin{equation}\label{eq:strauss-ell}
\gamma(n,\ell;p)
:=
\left(n-1+\frac{\ell}{1+\ell}\right)p^2
-
\left(n+1-\frac{3\ell}{1+\ell}\right)p
-2.    
\end{equation}
We prove that nontrivial, nonnegative, compactly supported data generate
finite-time blow-up whenever $\gamma(n,\ell;p)<0$. Note that \eqref{eq:strauss-ell} reduces to \eqref{eq:strauss-polinomio} when $\ell=0$. Therefore, when the coeficient of $p^2$ is positive, this polinomial has only one positive root and hence we denote it by $p_{\mathrm{Str}}(n,\ell)$. Moreover, this critical exponent coincides with the one appearing for the semilinear generalized Tricomi equation with power-type nonlinearity \cite{HeWiY20172, HeWiY2017,HeWiY2020,HeWiY2021, LinTu2019Tricomi,Su2021}. This suggests that the presence of the mass term in the oscillatory regime does not produce substantial effects on the blow-up threshold.

The method has two main components.  First, we construct a positive monotone temporal
solution of the separated adjoint equation \eqref{eq:intro-adjoint-ode}, with $\beta=0$,
and combine it with the spherical eigenfunction
$\varphi_\lambda$ as is standard.  This adjoint mode provides the lower estimate for the
nonlinear term and replaces the positivity that is lost at the level of the
space average.
Second, the complementary upper estimate is obtained by a different and genuinely oscillatory test. We test the weak formulation in convenient logarithmic time phases adapted to the oscillations of the Cauchy–Euler solution. On these phases, the oscillatory temporal
profile has a controlled sign and size, and the upper bound can be compared
with the lower bound obtained from the positive adjoint mode. This comparison naturally leads to the appearance of the  Strauss-type polynomial $\gamma(n,\ell;p)$ defined in \eqref{eq:strauss-ell}.

To the best of our knowledge, the use of logarithmic phase-localized
test functions for estimating solutions has not previously been used in the
blow-up theory of wave equations.  This is the main
novelty of the paper and the mechanism that allows us to treat the
oscillatory regime $\mu^2>1/4$.  The borderline case
$\gamma(n,\ell;p)=0$ remains open.

\medskip

The paper is organized as follows. In \Cref{sec:preliminaries}, we fix the notation, introduce the notion of finite-propagation weak solution and state the main blow-up theorem with its consequences. In \Cref{sec:lower-bound}, we present a positive and decreasing temporal profile for the separated adjoint equation and its large-time behavior. Combining this profile with an eigenfunction of the spherical Laplacian, we derive a monotonicity formula together with a pointwise lower bound for the nonlinear term. In \Cref{sec:log-shells}, we introduce the logarithmic time shells adapted to the oscillations of the Cauchy--Euler profile and prove the corresponding shell-localized lower and upper bounds. Finally, in \Cref{sec:proof-main}, we compare the lower and upper shell estimates and complete the proof of the main theorem, including the stated lifespan estimate. \Cref{appendixA} is devoted to the proof of the existence of the adjoint profile and the analysis of its asymptotic behavior.

\section{Notations and Main Results}\label{sec:preliminaries}

This section introduces the notation, definitions, and main results of the paper. We first fix the auxiliary quantities associated with the time-dependent propagation speed, the weighted Strauss-type polynomial, and the spatial test functions employed in the blow-up argument. We then formulate the notion of finite-propagation weak solution adapted to the weighted problem and state the main nonexistence and lifespan result, together with its consequences for the unweighted equation and for the corresponding damped model in the oscillatory regime.

\subsection{Definitions and Notations}

We begin by collecting the notation and auxiliary objects that will be used throughout the paper.  These include the Strauss-type polynomials governing the critical ranges of the exponent, the spatial and temporal weights entering the test-function method, and some standard conventions for comparison estimates.

\begin{enumerate}[label=$(\mathrm{D}\arabic*)$]
    \item (Weighted Strauss polynomial) For $n,\ell\in\R$ with $\ell\neq-1$ and for any $\alpha\ge 0$ we define the weighted Strauss-type polynomial on $p$ by\begin{equation}\label{eq:gamma}
\gamma_\alpha(n,\ell;p)
=
\left(n-1+\frac{\ell}{1+\ell}\right)p^2
-
\left(
n+1-\frac{3\ell+2\alpha}{1+\ell}
\right)p
-2.
\end{equation}

    \item \label{def:smooth_cutoff}(Smooth cutoff function) Let $\zeta\in C_c^{\infty}(\R^n)$ any function such that $0\le \zeta\le1$, $\zeta\equiv1$ on $B_1$ and $\supp\zeta\subset B_2$, and, for any $R>0$, set
\[\zeta_R(x):=\zeta\!\left(\frac{x}{R}\right)\]

    \item (Velocity function) We set
\[
a(t):=(1+t)^\ell
\quad \text{and} \quad
A(t):=\int_0^t a(s)\dd s
=
\frac{(1+t)^{\ell+1}-1}{\ell+1}.
\]
Since $-1<\ell<0$, we have $a(t)>0$, $a'(t)<0$, and $A(t)\to+\infty$ as $t\to+\infty$.

    \item\label{def:varphi_lambda} (Yordanov-Zhang spherical eigenfunction) For $\lambda>0$, let $\varphi_\lambda:\R^n\to\R$ defined by
\[
\varphi_\lambda(x):=\int_{\mathbb S^{n-1}}e^{\lambda x\cdot\omega}\dd S(\omega),
\]
with the convention $\mathbb S^0=\{-1,1\}$ when $n=1$. 

    \item \label{def:H} For a fixed $p>1$ and $u\in L^p_{\loc}([0,T)\times \R^n)$, we write \[
H(t):=\int_{\R^n}|u(x,t)|^p\dd x.
\]

    \item Given two nonnegative quantities $X$ and $Y$, we write
\[
X\asymp Y
\]
if there exist positive constants $c$ and $C$, independent of the relevant variables, such that
\[
c\,Y\le X\le C\,Y.
\]
Whenever needed, the dependence of the implicit constants are specified explicitly. 

\end{enumerate}

\begin{remark}\label{rem:weak-class}
The condition  $u\in L^p_{\loc}([0,T)\times \R^n)$ ensures that $H(t)$ is finite for almost every $t\in[0,T)$ because of the support condition \eqref{eq:finite-speed-weak}.  All pointwise estimates involving $H(t)$ are understood for a.e. $t$.
\end{remark}

\subsection{Main Results}

In this paper, we investigate the finite-time blow-up of solutions to the Tricomi equation with scale-invariant mass and a weighted power-type nonlinearity
\begin{equation}\label{eq:main2}
\begin{cases}
u_{tt}-(1+t)^{2\ell}\Delta u+\dfrac{\mu^{2}}{(1+t)^{2}}u
=(1+t)^{-\alpha}|u|^{p}, 
& (x,t)\in \mathbb{R}^{n}\times(0,T),\\[0.4em]
u(x,0)=\eps f(x), \quad u_{t}(x,0)=\eps g(x), 
& x\in \mathbb{R}^{n}.
\end{cases}
\end{equation}
Here $n\geq 1$, $-1<\ell<0$, $\mu^{2}>\frac14$, $p>1$, and $\eps>0$. Although our primary interest is the Cauchy problem \eqref{eq:intro-main}, in order to establish the blow-up result we work in a slightly more general setting and allow a weighted nonlinearity with parameter $\alpha\in\mathbb{R}$. As a consequence, we also obtain a finite-time blow-up result for an open case related to the Cauchy problem associated to \eqref{eq:intro-general-damping-mass} (see \cite{Palmieri2025}).

\begin{definition}[Finite-propagation weak solutions]\label{def:weak}
Let $\mu,\ell,\alpha\in\R$ with $\ell\neq -1$ and  $\alpha\ge0$. Let $T\in(0,\infty]$.  We say that $u$ is a weak solution on $[0,T)$ of \eqref{eq:main2}
if:
\begin{enumerate}[label = $(\roman*)$]
    \item $u\in C^1([0,T);L^1_{\loc}(\R^n))\cap L^p_{\loc}([0,T)\times\R^n)$;
    \item there exists $R_0>0$ such that
$\supp f\cup \supp g\subset B_{R_0}$,
and, for all $t\in[0,T)$,
\begin{equation}\label{eq:finite-speed-weak}
\supp u(\cdot,t)\cup \supp u_t(\cdot,t)\subset B_{R_0+A(t)};
\end{equation}
    \item for every $\phi\in C_c^\infty(\R^n\times [0,T))$,
\begin{equation}\label{eq:weak-formulation}
    \begin{split}
        \int_0^T&\!\!\int_{\R^n}u(x,t)\left(
 \phi_{tt}(x,t)-(1+t)^{2\ell}\Delta\phi(x,t)
 +\frac{\mu^2}{(1+t)^2}\phi(x,t)
 \right)\dd x\dd t \\
&=\int_0^T\!\!\int_{\R^n}(1+t)^{-\alpha}|u(x,t)|^p\phi(x,t)\dd x\dd t
 +\eps\int_{\R^n}\bigl(g(x)\phi(x,0)-f(x)\phi_t(x,0)\bigr)\dd x.
    \end{split}
\end{equation}
\end{enumerate}
\end{definition}

Let us state the main result of the paper:

\begin{theorem}\label{thm:lifespan-weighted}
Assume $n\ge1$, $-1<\ell<0$, $\mu^2>\frac14$, $\alpha\ge0$ and $p>1$. Let
\[
f,g\in C_c^\infty(\R^n),
\quad
f\ge0,
\quad
g\ge0,
\quad \text{and} \quad
(f,g)\not\equiv(0,0).
\]
Assume that
\[
\gamma_\alpha(n,\ell;p)<0.
\]
Then, no global weak solution of \eqref{eq:main2} with initial data $(\varepsilon f,\varepsilon g)$ exists for any $\varepsilon>0$. 
Moreover, there exist $\varepsilon_0>0$ and $C>0$, independent of $\varepsilon$, such that
for every $0<\varepsilon\le\varepsilon_0$, if the weak solution is defined on $[0,T_\eps)$, then the lifespan $T_\eps$ must satisfy
\begin{equation}\label{eq:lifespan-weighted}
T_\varepsilon
\le
C\varepsilon^{
\frac{2p(p-1)}
{(1+\ell)\gamma_\alpha(n,\ell;p)}
}.
\end{equation}
\end{theorem}

A first consequence of \Cref{thm:lifespan-weighted} is an explicit
description of parameter ranges in which the polynomial is
negative. In particular, when the propagation speed decreases sufficiently
fast, the blow-up condition holds for every $p>1$; outside this range, it
holds below the corresponding Strauss-type exponent.

\begin{corollary}\label{cor:alpha-zero-strongly-negative-ell}
Assume $\alpha=0$, $n\ge1$, and $\mu^2>\frac14$. If any of the following hypotheses hold

\begin{enumerate}[label=$(H\arabic*)$]
    \item \label{H1}$-1<\ell\le -1+\frac1n$ and $p>1$;

    \item \label{H2} $-1+\frac1n< \ell <0$ and $1<p<p_{\mathrm{Str}}(n,\ell)$; 
\end{enumerate}
then, there is no global weak solution of \eqref{eq:main2} with nontrivial nonnegative smooth compactly supported initial data.
\end{corollary}

\begin{proof}
For $\alpha=0$, the weighted Strauss-type polynomial becomes
\[
\gamma_0(n,\ell;p)
=
\left(
n-1+\frac{\ell}{1+\ell}
\right)p^2
-
\left(
n+1-\frac{3\ell}{1+\ell}
\right)p
-2.
\]
Equivalently,
\[
(1+\ell)\gamma_0(n,\ell;p)
=
\bigl(n(1+\ell)-1\bigr)p^2
-
\bigl((n+1)(1+\ell)-3\ell\bigr)p
-2(1+\ell).
\]
By assumption \ref{H1}, $n(1+\ell)-1<0$.
Moreover, since $\ell<0<1+\ell$, we have
\[
(n+1)(1+\ell)-3\ell>0.
\] Hence, for every $p>1$,
\[
\bigl(n(1+\ell)-1\bigr)p^2
-
\bigl((n+1)(1+\ell)-3\ell\bigr)p
-2(1+\ell)<0
\] and, 
consequently, for every $p>1$,
$\gamma_0(n,\ell;p)<0$.
Therefore, the blow-up condition of \Cref{thm:lifespan-weighted} is
automatically satisfied for every $p>1$.

Now, under \ref{H2}, we note that the the leading coefficient of $\gamma_0(n,\ell;p)$ is positive, thus $\gamma_0(n,\ell;p)\to+\infty$ as $p\to+\infty$. Since $\gamma_0(n,\ell;1) = -6<0$, there exists a root greater than 1. This is precisely $p_{\mathrm{Str}}(n,\ell)$. Therefore, if $1<p<p_{\mathrm{Str}}(n,\ell)$, we have $\gamma_0(n,\ell;1)<0$ and the blow-up follows from \Cref{thm:lifespan-weighted}.
\end{proof}

\begin{remark}
When $n=1$, condition \ref{H1} becomes $-1<\ell <0$, and therefore the above corollary gives
finite-time blow-up for every $p>1$.
\end{remark}

Next, we complement the blow-up result of \cite{Palmieri2025}, obtained
for $\delta>0$, to the oscillatory regime $\delta<0$. While in the
non-oscillatory regime the heat-like behavior still affects the blow-up condition,
in the oscillatory regime the heat-like phenomenon disappears and only the
wave-like behavior remains. In particular, the blow-up condition is
governed by the Strauss-type exponent. Moreover, the upper bound for the lifespan
remains the same as in the non-oscillatory case.

\begin{corollary}\label{cor:EPDT-equation}
Assume that $n\ge1$, $-1<\ell<0$, $\beta\ge0$, $\mu\in\mathbb R$,
$p>1$ are such that and $\delta:=(1-\beta)^2-4\mu^2<0$.
Let $w$ be a finite-propagation weak solution of
\eqref{eq:intro-general-damping-mass} with initial data
\[
w(x,0)=\varepsilon f(x)
\quad \text{and} \quad
w_t(x,0)=\varepsilon g(x),
\]
where
\[
f,g\in C_c^\infty(\mathbb R^n),
\quad
f\ge0,\quad g\ge0,\quad \text{and} \quad (f,g)\not\equiv(0,0).
\]
We have the following:
\begin{enumerate}[label=$(\roman*)$]
    \item If $n(1+\ell)+\beta \le 1$, then no global weak solution exists for any $p>1$ and any $\varepsilon>0$;
    \item If $n(1+\ell)+\beta > 1$ and $1<p<p_{\mathrm{Str}}\!\left(n+\frac{\beta}{1+\ell},\ell\right)$, then no global weak solution exists for any $\varepsilon>0$.
\end{enumerate}
In either case, there
exist $\varepsilon_0>0$ and $C>0$, independent of $\varepsilon$, such that
for every $0<\varepsilon\le\varepsilon_0$, the lifespan $T_\varepsilon$ of
$w$ satisfies
\[
T_\varepsilon
\le
C\varepsilon^{
\frac{2p(p-1)}
{(1+\ell)\gamma_{\alpha_\beta}(n,\ell;p)}
}
\quad \text{with}\quad\alpha_\beta = \frac{\beta}{2}(p-1).
\]

\end{corollary}

\begin{proof}


Defining the change of variables
$u(x,t):=(1+t)^{\beta/2}w(x,t)$ and replacing it in \eqref{eq:intro-general-damping-mass}, we obtain that $u$ satisfies

\[
\begin{cases}\displaystyle
    u_{tt}-(1+t)^{2\ell}\Delta u
+\frac{\widetilde{\mu}^2}{(1+t)^2}u
=
(1+t)^{-\alpha_\beta}|u|^p,& (x,t)\in\mathbb R^n\times(0,T)\\[0.8em]
u(x,0)=\varepsilon f(x),
\qquad
u_t(x,0)=\varepsilon\left(g(x)+\frac{\beta}{2}f(x)\right), & x\in\R^n
\end{cases}
\]
where 
\[
\alpha_\beta:=\frac{\beta}{2}(p-1) \quad \text{and}\quad \widetilde{\mu}^2 := \mu^2+\frac{\beta(2-\beta)}4.\]

Since $\beta\ge0$, $f\ge0$, $g\ge0$, and $(f,g)\not\equiv(0,0)$,
the transformed initial data are nonnegative, compactly supported, and not both
identically zero. Also, multiplication by the positive factor
$(1+t)^{\beta/2}$ preserves the finite-propagation support condition and is
invertible on every finite time interval.
Since $\delta=(1-\beta)^2-4\mu^2<0$, we have
\[
\widetilde{\mu}^2
=
\mu^2+\frac{\beta(2-\beta)}4
=
\frac{1-\delta}{4}>\frac14.
\]
Thus $\delta<0$ implies $\nu^2>1/4$, so \Cref{thm:lifespan-weighted}
applies to $u$, provided
\[
\gamma_{\alpha_\beta}(n,\ell;p)<0.
\]
The transformation $u=(1+t)^{\beta/2}w$ is bijective on every interval
$[0,T)$; hence, $u$ and $w$ have the same lifespan. Therefore, the
nonexistence of global solutions and the stated lifespan estimate for $w$
follow from \Cref{thm:lifespan-weighted}.
Since 
\begin{equation}\label{eq:strauss-polinomio-deslocado}
    \begin{split}
\gamma_{\alpha_\beta}(n,\ell;p)
&=
\left(n-1+\frac{\ell}{1+\ell}\right)p^2
-
\left(
n+1-\frac{3\ell+\beta(p-1)}{1+\ell}
\right)p
-2 \\
&=
\left(n+\frac{\beta}{1+\ell}-1+\frac{\ell}{1+\ell}\right)p^2
-
\left(
n+\frac{\beta}{1+\ell}+1-\frac{3\ell}{1+\ell}
\right)p
-2.
\end{split}
\end{equation}
We can see that, under the hypothesis that $n(1+\ell)+\beta\le 1$, and since $\beta\ge 0$, all terms in parenthesis in \eqref{eq:strauss-polinomio-deslocado} are positive. Then, for any $p>0$, $\gamma_{\alpha_\beta}(n,\ell;p)<-2<0$. This proves $(i)$.

For $(ii)$, the assumption $n(1+\ell)+\beta>1$ implies that the coefficient of $p^2$ in
$\gamma_{\alpha_\beta}(n,\ell;p)$ is positive. Moreover,
\[
\gamma_{\alpha_\beta}(n,\ell;1)=-\frac{4}{1+\ell}<0 \quad \text{ and }\quad \gamma_{\alpha_\beta}(n,\ell;0)=-2<0.
\]
Then, $\gamma_{\alpha_\beta}(n,\ell;\cdot)$
has exactly one positive root. Note that the polinomial \eqref{eq:strauss-polinomio-deslocado} is \eqref{eq:strauss-ell} shifted by $\frac\beta{1+\ell}$. Therefore, this positive root is the Strauss root $p_{\mathrm{Str}}\!\left(n,\ell\right)$ shifted by $\frac{\beta}{1+\ell}$. Consequently,
\[
p>1
\quad\text{and}\quad
\gamma_{\alpha_\beta}(n,\ell;p)<0
\]
if and only if
\[
1<p<p_{\mathrm{Str}}\!\left(n+\frac{\beta}{1+\ell},\ell\right).\qedhere
\]
\end{proof}

\begin{remark}
Since $n\ge1$ and $-1<\ell<0$, one has $n(1+\ell)>0$. Thus, the
condition $n(1+\ell)+\beta\le 1$ together with $\beta\ge0$ implies $0\le \beta\le 1-n(1+\ell)<1$. Therefore, item $(i)$ can occur only for $0\le\beta<1$. If $n(1+\ell)>1$, this case is empty.
\end{remark}




\section{A weighted lower bound for the nonlinear term}\label{sec:lower-bound}

Our goal is to derive a pointwise lower bound for the nonlinear term $H(t)$, defined in \ref{def:H}, where the key auxiliary ingredient is the existence of a positive decreasing temporal factor solving the corresponding separated adjoint equation.

Positive adjoint profiles of this type are a standard device in blow-up arguments based on weighted test functions. For instance \cite{YZ2006} introduced a related elliptic-spatial test function method for the classical semilinear wave equation, later adapted by  \cite{HHHT22,FanLaiTakamura25,Palmieri2025, WakaYor17,WakaYor19} to time-dependent speeds and related damped and mass settings.

The estimates for the adjoint profile are frequently obtained through its representation in terms of modified Bessel functions. More precisely, after a suitable change of variables, the adjoint equation can be brought into a
modified Bessel equation, and the decaying solution is then expressed by means
of a modified Bessel function of the second kind with order proportional to $\sqrt{\delta}$. In the present
regime the same reduction remains possible, but it leads to a purely
imaginary order. Since only a limited set of properties of the adjoint profile is needed in the sequel,
we choose instead to prove the required bounds directly by elementary ODE methods. This has the advantage
of keeping the argument essentially self-contained and of avoiding unnecessary
use of special-function theory. For convenience, this proof is presented in \Cref{appendixA}.

\begin{proposition}\label{prop:adjoint-profile}
	Assume $-1<\ell<0$ and $\mu^2>1/4$. For every $\lambda>|\mu|$, there exists a function
	$m_\lambda: [0,\infty) \longrightarrow \R$ of class $C^2$
	such that, for all $t\ge0$, there hold $m_\lambda(t)>0$,
	$m_\lambda'(t)<0$, and
	\begin{equation}\label{eq:m-eq}
		m_\lambda''(t)
		=
		\biggl(\lambda^2(1+t)^{2\ell}-\frac{\mu^2}{(1+t)^2}\biggr) \, m_\lambda(t).
	\end{equation}
	Moreover, there exist constants $c>0$, $C>0$ and $t_0\ge 1$, depending at most on $\lambda,\ell$, and $\mu$, such that, for all $t\ge t_0$,
	\begin{equation}\label{eq:m-bounds}
		c\,a(t)^{-1/2}e^{-\lambda A(t)}
		\le
		m_\lambda(t)
		\le
		C\,a(t)^{-1/2}e^{-\lambda A(t)}
	\end{equation}
	and
	\begin{equation}\label{eq:mprime-bounds}
		c\,a(t)^{1/2}e^{-\lambda A(t)}
		\le
		-m_\lambda'(t)
		\le
		C\,a(t)^{1/2}e^{-\lambda A(t)}.
	\end{equation}
\end{proposition}

\vspace{1em}

Now, we use the adjoint function from the previous proposition to construct a separated variables solution to the linear part of \eqref{eq:main2}.  

Fix $\lambda>|\mu|$, and let $m_\lambda$ be given by \Cref{prop:adjoint-profile}. Define
\[
\Phi_\lambda(x,t):=m_\lambda(t)\varphi_\lambda(x),
\]
where $\varphi_\lambda$ is the spherical eigenfunction defined in \ref{def:varphi_lambda}. Since $\Delta\varphi_\lambda=\lambda^2\varphi_\lambda$ and \eqref{eq:m-eq} holds, we have
\begin{equation}\label{eq:adjoint}
\Phi_{\lambda,tt}-(1+t)^{2\ell}\Delta\Phi_\lambda+\frac{\mu^2}{(1+t)^2}\Phi_\lambda=0
\quad\text{in }\R^n\times(0,\infty).
\end{equation}
Let us define the function
\[
Q_\lambda(t)
:=
\int_{\R^n}
\bigl[
m_\lambda(t)u_t(x,t)-m_\lambda'(t)u(x,t)
\bigr]\varphi_\lambda(x) \dd x.
\]

The next lemma shows that testing the equation against the positive adjoint mode $\Phi_\lambda$ yields a monotonicity formula for the Wronskian-type quantity $Q_\lambda$.

\begin{lemma}\label{lem:Q-monotone}
Let $u$ be a weak solution of \eqref{eq:main2} on $[0,T)$ and fix $\lambda>|\mu|$. Then 
$Q_\lambda \in W^{1,1}_{\loc}((0,T))$. Also, for every $0\le s\le t<T$,
\begin{equation}\label{eq:Q-weak-identity}
Q_\lambda(t)-Q_\lambda(s)
=
\int_s^t\int_{\R^n}(1+\tau)^{-\alpha}|u(x,\tau)|^p\Phi_\lambda(x,\tau)\dd x\dd \tau.
\end{equation}
In particular,
\[
Q_\lambda'(t)=(1+t)^{-\alpha}\int_{\R^n}|u(x,t)|^p\Phi_\lambda(x,t)\dd x,
\text{ for a.e. }t\in(0,T),
\]
and $Q_\lambda$ is nondecreasing on $[0,T)$.
\end{lemma}

\begin{proof}
By \eqref{eq:finite-speed-weak}, for every fixed $t\in[0,T)$ the functions
$u(t,\cdot)$ and $u_t(t,\cdot)$ are compactly supported. Since
$\varphi_\lambda$ is smooth, the quantity $Q_\lambda(t)$ is well defined.

Let $\zeta\in C_c^{\infty}(\R^n)$ be a smooth cutoff function as in \ref{def:smooth_cutoff}. For $R>0$, define
\[
\Phi_{\lambda,R}(x,t):=\zeta_R(x)\Phi_\lambda(x,t).
\]
Then, $\Phi_{\lambda,R}$ is smooth and compactly supported in the spatial
variable. Using \eqref{eq:adjoint}, we compute
\begin{equation}\label{eq:Phi_lambda_R}
\Phi_{\lambda,R,tt}
-(1+t)^{2\ell}\Delta\Phi_{\lambda,R}
+\frac{\mu^2}{(1+t)^2}\Phi_{\lambda,R}
=
-(1+t)^{2\ell}m_\lambda(t)
\Bigl(
2\,\nabla\zeta_R\cdot\nabla\varphi_\lambda+\varphi_\lambda\Delta\zeta_R
\Bigr).
\end{equation}

\noindent Next, define
\[
Q_{\lambda,R}(t)
:=
\int_{\R^n}
\bigl(
m_\lambda(t)u_t(x,t)-m_\lambda'(t)u(x,t)
\bigr)\varphi_\lambda(x)\zeta_R(x)\dd x.
\]
For $0<s<t<T$, we let $\theta\in C_c^{\infty}((s,t))$ and we use
\[
\phi(x,\tau)=\theta(\tau)\Phi_{\lambda,R}(x,\tau) \in C_c^{\infty}(\R^n\times[0,\infty))
\]
as a test function in \eqref{eq:weak-formulation}: since $\Delta\phi
=
\theta(\tau)\Delta\Phi_{\lambda,R}$ and 
\[
\phi_{\tau\tau}
=
\theta''(\tau)\Phi_{\lambda,R}
+
2\theta'(\tau)\partial_\tau\Phi_{\lambda,R}
+
\theta(\tau)\partial_{\tau\tau}\Phi_{\lambda,R},
\]
the left-hand side of \eqref{eq:weak-formulation} becomes
\begin{align*}
\int_0^T\!\!\int_{\R^n}
u(x,\tau)
&\left(
\phi_{\tau\tau}(x,\tau)
-
(1+\tau)^{2\ell}\Delta\phi(x,\tau)
+
\frac{\mu^2}{(1+\tau)^2}\phi(x,\tau)
\right)\dd x\dd \tau\\
& = \int_s^t \!\int_{\R^n}
u(x,\tau)
\big[
\theta''(\tau)\Phi_{\lambda,R}
+
2\theta'(\tau)\partial_\tau\Phi_{\lambda,R}
\big]\dd x\dd \tau
\\
&\quad \
+
\int_s^t
\theta(\tau)
\int_{\R^n}
u(x,\tau)
\left(
\partial_{\tau\tau}\Phi_{\lambda,R}
-
(1+\tau)^{2\ell}\Delta\Phi_{\lambda,R}
+
\frac{\mu^2}{(1+\tau)^2}\Phi_{\lambda,R}
\right)\dd x\dd \tau .
\end{align*}
Let us first treat the terms containing derivatives of $\theta$. Set
\[
G_R(\tau)
:=
\int_{\R^n}
u(x,\tau)\varphi_\lambda(x)\zeta_R(x)\dd x \quad \text{so that} \quad 
G_R'(\tau)
=
\int_{\R^n}
u_t(x,\tau)\varphi_\lambda(x)\zeta_R(x)\dd x
\]
and then
\[
Q_{\lambda,R}(\tau)
=
m_\lambda(\tau)G_R'(\tau)
-
m_\lambda'(\tau)G_R(\tau).
\]
We compute
\begin{align*}
\int_s^t\! \int_{\R^n}
u(x,\tau)
\big[
\theta''(\tau)\Phi_{\lambda,R}
+
2\theta'(\tau)\partial_\tau\Phi_{\lambda,R}
\big]\dd x\dd \tau
&=
\int_s^t
\big[
\theta''(\tau)m_\lambda(\tau)G_R(\tau)
+
2\theta'(\tau)m_\lambda'(\tau)G_R(\tau)
\big]\dd \tau\\
&=
-\int_s^t
\theta'(\tau)
\big[
m_\lambda(\tau)G_R'(\tau)
-
m_\lambda'(\tau)G_R(\tau)
\big]\dd \tau
\\
&=
-\int_s^t Q_{\lambda,R}(\tau)\theta'(\tau)\dd \tau .
\end{align*}
where the second equality is obtained by an integration by parts in the first term. 

For the other integral we use the equation \eqref{eq:Phi_lambda_R} satisfied by $\Phi_{\lambda,R}$. The weak formulation with test function $\phi$ becomes
\begin{equation}\label{eq:Q-weak-cutoff}
\begin{split}
    -\int_s^t Q_{\lambda,R}(\tau)\theta'(\tau)\dd \tau
&-\int_s^t
\theta(\tau)(1+\tau)^{2\ell}m_\lambda(\tau)
\int_{\R^n}
u(x,\tau)
\Bigl(
2\nabla\zeta_R\cdot\nabla\varphi_\lambda
+
\varphi_\lambda\Delta\zeta_R
\Bigr)\dd x\dd \tau \\
&=
\int_s^t
\theta(\tau)(1+\tau)^{-\alpha}
\int_{\R^n}
|u(x,\tau)|^p\Phi_{\lambda,R}(x,\tau)\dd x\dd \tau
\end{split}
\end{equation}

\noindent Now choose $R>2\bigl(R_0+A(t)\bigr)$. Then, for every $\tau\in[s,t]$,
\[
\supp u(\tau,\cdot)\subset B_{R_0+A(\tau)}\subset B_{R_0+A(t)}\subset B_{R/2},
\]
and, hence, on the support of $u(\tau,\cdot)$, we have $\zeta_R\equiv1$, $\nabla\zeta_R\equiv0$ and  $\Delta\zeta_R\equiv0$. In particular, the error term in \eqref{eq:Q-weak-cutoff} vanishes. Furthermore,
$Q_{\lambda,R}(\tau)=Q_\lambda(\tau)$ for all $\tau\in[s,t]$,
and $\Phi_{\lambda,R}(\cdot,\tau)=\Phi_\lambda(\cdot,\tau)$
 on $\supp u(\tau,\cdot)$, for all $\tau\in[s,t]$.
Thus,
\[
-\int_s^t Q_\lambda(\tau)\theta'(\tau)\dd \tau
=
\int_s^t\theta(\tau)(1+\tau)^{-\alpha}\int_{\R^n}|u(x,\tau)|^p\Phi_\lambda(x,\tau)\dd x\dd \tau,
\]
for every $\theta\in C_c^{\infty}((s,t))$.
This proves that
\[
Q_\lambda'(\tau)=(1+\tau)^{-\alpha}\int_{\R^n}|u(x,\tau)|^p\Phi_\lambda(x,\tau)\dd x
\] in $\mathcal D'((0,T))$
and, hence, almost everywhere. Since the right-hand side belongs to
$L^1_{\loc}((0,T))$, we conclude that
\[
Q_\lambda\in W^{1,1}_{\loc}((0,T)).
\]
Integrating from $s$ to $t$ yields \eqref{eq:Q-weak-identity}. Since the
right-hand side is nonnegative, $Q_\lambda$ is nondecreasing on $(0,T)$. 
It remains to justify the monotonicity up to the initial time.  
We claim that $Q_\lambda$ admits a continuous extension to $t=0$, with
\[
Q_\lambda(0)
=
\varepsilon
\int_{\mathbb R^n}
\bigl(
m_\lambda(0)g(x)-m_\lambda'(0)f(x)
\bigr)\varphi_\lambda(x)\dd x .
\]
Indeed, fix $\delta\in(0,T)$. By the finite-propagation condition, we have, for every $t\in[0,\delta]$,
\[
\supp u(\cdot,t)\cup\supp u_t(\cdot,t)
\subset B_{R_0+A(\delta)}.
\]
Since $\varphi_\lambda\in L^\infty(B_{R_0+A(\delta)})$ and
$u\in C^1([0,T);L^1_{\rm loc}(\mathbb R^n))$, we have, as $t\downarrow0$,
\[
u(t)\to \varepsilon f
\quad \text{and} \quad
u_t(t)\to \varepsilon g
\quad\text{in }L^1(B_{R_0+A(\delta)}),
\] that, together with the continuity of $m_\lambda$ and $m_\lambda'$, gives
\[
Q_\lambda(t)\to Q_\lambda(0)
=
\varepsilon\int_{\mathbb R^n}
\bigl(m_\lambda(0)g(x)-m_\lambda'(0)f(x)\bigr)
\varphi_\lambda(x)\dd x .
\]

Therefore, the identity \eqref{eq:Q-weak-identity}, already proved for $0<s<t<T$,
can be passed to the limit $s\downarrow0$. Since the integrand is
nonnegative and locally integrable in time, the monotone convergence theorem
yields
\[
Q_\lambda(t)-Q_\lambda(0)
=
\int_0^t\int_{\mathbb R^n}
(1+\tau)^{-\alpha}|u(x,\tau)|^p\Phi_\lambda(x,\tau)\dd x\dd \tau
\ge0.
\]
Consequently,
$Q_\lambda(t)\ge Q_\lambda(0)$,
for every $t\in[0,T)$,
and $Q_\lambda$ is nondecreasing on the interval $[0,T)$.
\end{proof}

The preceding monotonicity formula shows that $Q_\lambda$ can only increase. We now use the sign assumptions on the initial data to show that it starts from a strictly positive value.

\begin{lemma}\label{lem:Q-positive}
Let $u$ be a weak solution of \eqref{eq:main2} on $[0,T)$. There exists $c_0=c_0(\lambda,f,g)>0$ such that
$Q_\lambda(t)\ge c_0\eps$, for all $t\in[0,T)$.
\end{lemma}

\begin{proof}
Since $m_\lambda(0)>0$, $-m_\lambda'(0)>0$,  $\varphi_\lambda(x)>0$, and the initial conditions $f,g$ are both non-negative with $(f,g)\not\equiv(0,0)$, it follows that 
\[c_0:=\int_{\mathbb R^n}
\bigl(m_\lambda(0)g(x)-m_\lambda'(0)f(x)\bigr)
\varphi_\lambda(x)\dd x>0.\]
Moreover, we have 
$Q_\lambda(0)= c_0\eps$.
Since $Q_\lambda$ is nondecreasing on $[0,T)$, the lemma follows.
\end{proof}

The positivity of $Q_\lambda$ yields a quantitative lower bound for the weighted spatial average of the solution.

\begin{proposition}\label{prop:F-growth}
Let $u$ be a weak solution of \eqref{eq:main2} on $[0,T)$. For $\lambda>|\mu|$, define the weighted spatial average
\[
F_\lambda(t):=\int_{\R^n}u(x,t)\varphi_\lambda(x)\dd x.
\]
Then, there exist constants $t_0\ge1$ and $C_0>0$ such that, for all $t\in[t_0,T)$,
\[
F_\lambda(t)\ge C_0\,\eps\,a(t)^{-1/2}e^{\lambda A(t)}.
\]
\end{proposition}

\begin{proof}
By the finite propagation property, the integral defining $F_\lambda(t)$ is finite and so it is well defined. By definition,
\[
Q_\lambda(t)
=
m_\lambda(t)F_\lambda'(t)-m_\lambda'(t)F_\lambda(t)
=
m_\lambda(t)^2\biggl(\frac{F_\lambda(t)}{m_\lambda(t)}\biggr)'.
\]
Hence, by \Cref{lem:Q-positive},
\begin{equation}\label{eq:diff-F/m}
\biggl(\frac{F_\lambda(t)}{m_\lambda(t)}\biggr)'
=
\frac{Q_\lambda(t)}{m_\lambda(t)^2}
\ge
\frac{c_0\eps}{m_\lambda(t)^2}.    
\end{equation}
Integrating from $0$ to $t$, we obtain
\[
\frac{F_\lambda(t)}{m_\lambda(t)}
\ge
\frac{F_\lambda(0)}{m_\lambda(0)}
+
c_0\eps\int_0^t \frac{1}{m_\lambda(s)^2} \dd s.
\]
Since $F_\lambda(0)=\eps\int f\varphi_\lambda\ge0$, this yields
\[
F_\lambda(t)\ge c_0\eps\,m_\lambda(t)\int_0^t \frac{1}{m_\lambda(s)^2} \dd s.
\]
By \eqref{eq:m-bounds}, for all $t$ large enough and $s\in[t_0,t]$, we have
\[
m_\lambda(t)\ge c\,a(t)^{-1/2}e^{-\lambda A(t)}
\quad \text{and} \quad
m_\lambda(s)^{-2}\ge c\,a(s)e^{2\lambda A(s)}.
\]
Then,
\[
F_\lambda(t)
\ge c_0\eps\,m_\lambda(t)\int_{t_0}^t\frac{1}{m_\lambda(s)^2} \dd s
\ge
c\,\eps\,a(t)^{-1/2}e^{-\lambda A(t)}
\int_{t_0}^t a(s)e^{2\lambda A(s)}\dd s.
\]
Recall $A'(s)=a(s)$, whence
\[
F_\lambda(t)
\ge
\frac{c\eps}{2\lambda}\,
a(t)^{-1/2}e^{\lambda A(t)}
\Bigl(1-e^{-2\lambda(A(t)-A(t_0))}\Bigr).
\]
Now, since $A(t)\to+\infty$ as $t\to+\infty$, there exists $t_1\ge t_0$ such that, for every $t\ge t_1$,
\[
1-e^{-2\lambda(A(t)-A(t_0))}\ge \frac12\,;
\]
therefore, for all $t\ge t_1$,
\[
F_\lambda(t)
\ge
\frac{c\eps}{4\lambda}\,
a(t)^{-1/2}e^{\lambda A(t)}.
\]
Renaming $t_1$ as $t_0$, this ends the proof with
$C_0:=c/(4\lambda)$.
\end{proof}

In what follows, we also need the following standard pointwise estimate for $\varphi_\lambda$. This bound appears for instance in \cite[Lemma 2.2]{WakaYor19}; see also
\cite{HeWiY2017,HeWiY2021,Palmieri2025,YZ2006}. For the sake of completeness, we include the proof.

\begin{lemma}\label{lem:phi-bound}
Let $\varphi_\lambda$ be the function defined in \ref{def:varphi_lambda}. Then, there exists $C=C(n,\lambda)>0$ such that, for all $x\in\R^n$,
\begin{equation}\label{eq:phi-pointwise}
\varphi_\lambda(x)\le C(1+|x|)^{-\frac{n-1}{2}}e^{\lambda|x|}.
\end{equation} 
\end{lemma}

\begin{proof}
If $n=1$, then
\[
\varphi_\lambda(x)=e^{\lambda x}+e^{-\lambda x}\le 2e^{\lambda|x|},
\]
which is \eqref{eq:phi-pointwise} with $C=2$.

Assume now that $n=2$  and let $r:=|x|$. For $0\le r\le1$, we have $(1+r)^{-1/2}\ge 2^{-1/2}$ and so
\[
\varphi_\lambda(x)
\le
|\mathbb S^1|e^\lambda
\le
C_\lambda (1+r)^{-1/2}e^{\lambda r}.
\] If, on the other hand, $r\ge1$, by rotational invariance and the spherical-coordinate 
formula, we decompose the integral as
\[
\varphi_\lambda(x)
=
|\mathbb S^0|
\int_{-1}^{1}
e^{\lambda r s}(1-s^2)^{-1/2}\,\dd s =
I_-(r)+I_+(r),
\]
where
\[
I_-(r):=
\int_{-1}^{0}
e^{\lambda r s}(1-s^2)^{-1/2}\,\dd s
\quad \text{ and } \quad I_+(r):=
\int_{0}^{1}
e^{\lambda r s}(1-s^2)^{-1/2}\,\dd s.\]
On interval $[-1,0]$, since $e^{\lambda r s}\le1$ and $d_\lambda:=\inf_{r\ge1} r^{-1/2}e^{\lambda r}>0$, we have
\[
I_-(r)
\le
\int_{-1}^{0}(1-s^2)^{-1/2}\,\dd s
=
\frac{\pi}{2}
\le \frac{\pi}{2d_\lambda} \, r^{-1/2}e^{\lambda r}.
\]
Next,
on the interval $[0,1]$, observe $1-s^2=(1-s)(1+s)\ge 1-s$. Hence,
\[
I_+(r)
\le
\int_0^1 e^{\lambda r s}(1-s)^{-1/2}\,\dd s
\]
and the change of variables $y=\lambda r(1-s)$ implies
\[
I_+(r)
\le
e^{\lambda r}(\lambda r)^{-1/2}
\int_0^{\lambda r} e^{-y}y^{-1/2}\,\dd y
\le 
e^{\lambda r}(\lambda r)^{-1/2}\, \Gamma\left(\frac12\right) = \sqrt{\pi} \, \lambda^{-1/2} \, r^{-1/2} e^{\lambda r}.
\]
Combining the estimates for $I_-(r)$ and $I_+(r)$ in the case $r \ge 1$, we obtain
$\varphi_\lambda(x)
\le
C_\lambda r^{-1/2}e^{\lambda r}$.
Since $r^{-1/2}\le \sqrt2\,(1+r)^{-1/2}$ when $r\ge 1$, we have
\[
\varphi_\lambda(x)
\le
C_\lambda (1+r)^{-1/2}e^{\lambda r}
=
C_\lambda (1+|x|)^{-1/2}e^{\lambda |x|}.
\]
Together with the case $0\le r\le1$, this proves \eqref{eq:phi-pointwise} for $n=2$.

Finally, assume that $n\ge3$ and again we split the analysis into the cases $0\le r\le1$ and $r\ge1$, where $r:=|x|$. 
If $0\le r\le1$, we use the definition of $\varphi_\lambda(x)$ to obtain
\[
\varphi_\lambda(x) \le |\mathbb S^{n-1}|e^\lambda = |\mathbb S^{n-1}|e^\lambda 2^{\frac{n-1}{2}} 2^{-\frac{n-1}{2}}=: C(n,\lambda) \, 2^{-\frac{n-1}{2}} \le C (1+r)^{-\frac{n-1}{2}} \le C(1+r)^{-\frac{n-1}{2}}e^{\lambda r}.
\] If, on the other hand, $r\ge1$,
by rotational invariance, we may write $x=re_n$ and, then, the spherical-coordinate formula gives
\[
\varphi_\lambda(x)
=
|\mathbb S^{n-2}|
\int_{-1}^{1} e^{\lambda r s}(1-s^2)^{\frac{n-3}{2}}\dd s.
\] Since $1+s\le2$ for $s \in [-1,1]$ and $n \ge 3$, we note that
\[
(1-s^2)^{\frac{n-3}{2}}
=
(1+s)^{\frac{n-3}{2}}
(1-s)^{\frac{n-3}{2}} \le
C\,(1-s)^{\frac{n-3}{2}};
\]
hence,
\[
\varphi_\lambda(x)
\le
C\int_{-1}^{1} e^{\lambda r s}(1-s)^{\frac{n-3}{2}}\dd s.
\]
Now, the change of variables $y=\lambda r(1-s)$ gives
\begin{align*}
\varphi_\lambda(x)
&\le
C e^{\lambda r}(\lambda r)^{-\frac{n-1}{2}}
\int_{0}^{2\lambda r} e^{-y}y^{\frac{n-3}{2}}\dd y \\
&\le C \, \lambda^{-\frac{n-1}{2}} \, \Gamma \left(\frac{n-1}{2} \right) r^{-\frac{n-1}{2}} \, e^{\lambda r}.
\end{align*}
Therefore, in any case, we obtain
\[
\varphi_\lambda(x)\le C\, r^{-\frac{n-1}{2}}e^{\lambda r}.
\]
As before, this yields \eqref{eq:phi-pointwise}.
\end{proof}

In order to convert the weighted lower bound for $F_\lambda$ into a lower
bound for $H$, as defined in \ref{def:H}, we apply Hölder's
inequality on a ball containing the support of $u(t,\cdot)$. For this purpose,
we need the following $L^{p'}$ estimate for $\varphi_\lambda$ on balls.

\begin{lemma}\label{lem:phiLp}
Let $p'$ be the Hölder conjugate exponent of $p>1$. Then, there exists a constant $C=C(n,p,\lambda)>0$ such that, for every $\rho\ge1$,
\begin{equation}\label{eq:phiLp-bound}
\norm{\varphi_\lambda}_{L^{p'}(B_\rho)}^{p'}
\le
C\,\rho^{\,n-1-\frac{n-1}{2}p'}e^{p'\lambda\rho}.
\end{equation}
\end{lemma}

\begin{proof}
First, we prove the following elementary estimate: for $a\in\R$ and $b>0$, there exists $C_{a,b}>0$ such that, for every $\rho\ge1$,
\begin{equation}\label{eq:laplace-elem}
 \int_1^\rho r^a e^{b r}\,\dd r
 \le C_{a,b}\rho^a e^{b\rho}.
\end{equation}
For $1\le \rho\le2$, the left-hand side is bounded by
\[
M_{a,b}:=\int_1^2 r^a e^{b r}\,\dd r<\infty,
\]
whereas $\rho^a e^{b\rho}$ has a positive minimum on $[1,2]$.
Hence the desired estimate follows by choosing for example $C_{a,b} :=  (\min_{1\le\rho\le 2}\rho^ae^{b\rho})^{-1} M_{a,b}$.

If $\rho\ge2$, split the integral into the intervals $[1,\rho/2]$ and $[\rho/2,\rho]$. On the interval $[\rho/2,\rho]$, we have $r^a\le C\rho^a$, for all $a\in\R$ and, hence, this part is bounded by $C\rho^a e^{b\rho}$. On the other hand, on $[1,\rho/2]$, we have
\[
 \int_1^{\rho/2}r^a e^{b r}\,\dd r
 \le C(1+\rho^{\max\{a,0\}})e^{b\rho/2}
 \le C\rho^a e^{b\rho},
\]
the last inequality valid since $\rho^{-a}e^{-b\rho/2}$ is bounded for $\rho\ge2$ in the case $a<0$. This proves \eqref{eq:laplace-elem}.

\smallskip

Next, from \Cref{lem:phi-bound}, we know that, for all $x\in\R^n$,
\[
\varphi_\lambda(x)^{p'}
\le
C(1+|x|)^{-\frac{n-1}{2}p'}e^{p'\lambda|x|}.
\]
We integrate on $B_\rho$ and we use polar coordinates to obtain
\[
\norm{\varphi_\lambda}_{L^{p'}(B_\rho)}^{p'}
\le
C\int_0^\rho
r^{n-1}(1+r)^{-\frac{n-1}{2}p'}
e^{p'\lambda r}\,\dd r.
\]
The idea is now to apply \eqref{eq:laplace-elem} with $
a:=n-1-(n-1)p'/2$ and 
$b:=p'\lambda$.
We split the integral into the regions $0\le r\le1$ and $1\le r\le\rho$.
For the first part,
\[
\int_0^1
r^{n-1}(1+r)^{-\frac{n-1}{2}p'}
e^{p'\lambda r}\,\dd r
\le C.
\]
On the other hand, the function
$\rho\mapsto \rho^a e^{b\rho}$
is continuous, strictly positive on $[1,\infty)$, and tends to
$+\infty$ as $\rho\to+\infty$. So, it has a positive minimum on
$[1,\infty)$. Hence,
$C\le C\,\rho^a e^{b\rho}$, for every $\rho\ge1$.

For the second part, since $\rho\ge1$,
\[
\int_1^\rho
r^{n-1}(1+r)^{-\frac{n-1}{2}p'}
e^{p'\lambda r}\,\dd r
\le
C\int_1^\rho r^a e^{br}\,\dd r.
\]
By \eqref{eq:laplace-elem}, for every $\rho\ge1$,
\[
\int_1^\rho r^a e^{br}\,\dd r
\le
C_{a,b} \ \rho^a \, e^{b\rho}.
\]
Combining the two estimates, we conclude \eqref{eq:phiLp-bound}:
\[
\norm{\varphi_\lambda}_{L^{p'}(B_\rho)}^{p'}
\le
C\,\rho^a e^{b\rho}
=
C\,\rho^{\,n-1-\frac{n-1}{2}p'}e^{p'\lambda\rho}. \qedhere
\]
\end{proof}

Finally, we can obtain a pointwise lower bound estimate for the quantity $H$.

\begin{proposition}\label{prop:H-pointwise}
Let $u$ be a weak solution of \eqref{eq:main2} on $[0,T)$ and $H$ as defined in \ref{def:H}.
There exist constants $t_1\ge1$ and $C_1>0$, independent of $T$,
such that
\begin{equation}\label{eq:H-pointwise}
H(t)
\ge
C_1\,\eps^p\,
(1+t)^{-\frac{\ell p}{2}}
(1+A(t))^{-\frac{n-1}{2}(p-2)}
\end{equation}
for a.e. $t\in[t_1,T)$.
\end{proposition}

\begin{proof}
Let $t_0\ge1$ be the constant given by \Cref{prop:F-growth}.
Enlarging $t_0$, if necessary, we may also assume that
$R_0+A(t)\ge1$
 for all $t\ge t_0$,
so that \Cref{lem:phiLp} can be applied with
$\rho=R_0+A(t)$.

Set $t_1:=t_0$. If $T\le t_1$, there is nothing to prove. Thus, we may
assume that $T>t_1$.
For a.e. $t\in[t_1,T)$, \eqref{eq:finite-speed-weak} and Hölder's
inequality imply
\[
F_\lambda(t)
=
\int_{\R^n}u(x,t)\varphi_\lambda(x)\dd x
\le
H(t)^{1/p}
\norm{\varphi_\lambda}_{L^{p'}(B_{R_0+A(t)})}.
\]
As in \eqref{eq:diff-F/m}, we have that, for all $t\in[0,T)$,
\[\left(\frac{F_\lambda(t)}{m_\lambda(t)}\right)'
=
\frac{Q_\lambda(t)}{m_\lambda(t)^2}
\ge
\frac{c_0\eps}{m_\lambda(t)^2}
> 0.\]
Since $F_\lambda(0)\ge 0$, we obtain that $F_\lambda(t)> 0$, for all $t\in[0,T)$, and then
\[
H(t)
\ge
\frac{F_\lambda(t)^p}
{\norm{\varphi_\lambda}_{L^{p'}(B_{R_0+A(t)})}^{p}}.
\]
From \Cref{prop:F-growth}, there exists $t_1\ge 1$ such that, for all $t\in[t_1,T)$,
\[
F_\lambda(t)\ge C_0\,\eps\, a(t)^{-1/2}e^{\lambda A(t)}.
\]
Moreover, by \Cref{lem:phiLp} with $\rho=R_0+A(t)$,
\[
\norm{\varphi_\lambda}_{L^{p'}(B_{R_0+A(t)})}^{\,p}
\le
C\,(R_0+A(t))^{\,\frac{p}{p'}
\left(n-1-\frac{n-1}{2}p'\right)}
e^{p\lambda(R_0+A(t))}.
\]
Hence, for a.e. $t\in[t_1,T)$,
\[
H(t)
\ge
C\,\eps^p\,
a(t)^{-p/2}\,
(R_0+A(t))^{-\frac{p}{p'}
\left(n-1-\frac{n-1}{2}p'\right)}
e^{-p\lambda R_0}.
\]
Since
\[
\frac{p}{p'}=p-1
\qquad
\text{and}
\qquad
(p-1)\left(n-1-\frac{n-1}{2}p'\right)
=
\frac{n-1}{2}(p-2),
\]
we have, for a.e. $t\in[t_1,T)$,
\[
H(t)
\ge
C\eps^p
a(t)^{-p/2}
(R_0+A(t))^{-\frac{n-1}{2}(p-2)}.
\]

Now, $R_0>0$ is fixed and $A(t)\ge0$; then, the constants
$c_{R_0}=\min\{1,R_0\}$ and $C_{R_0}=R_0+1$ are such that, for all $t \ge 0$,
\[
c_{R_0}(1+A(t))
\le
R_0+A(t)
\le
C_{R_0}(1+A(t)).
\]
Thus, for any $\vartheta\in\R$, there exists $C=C(\vartheta,R_0)$ such that, for all $t \ge 0$,
\[
(R_0+A(t))^\vartheta
\ge
C(1+A(t))^\vartheta.
\]
Taking
$\vartheta=-(n-1)(p-2)/2$,
we obtain, for a.e. $t\in[t_1,T)$,
\[
H(t)
\ge
C\eps^p
a(t)^{-p/2}
(1+A(t))^{-\frac{n-1}{2}(p-2)}.
\]

\noindent
Finally, recall $a(t)=(1+t)^\ell$ and the proof is complete.
\end{proof}

\section{Phase-localized bounds on logarithmic shells}\label{sec:log-shells}

In the oscillatory regime, the temporal equation
\begin{equation}\label{eq:h-ode-ce}
    h''(t)+\frac{\mu^2}{(1+t)^2}h(t)=0
\end{equation}
has a real solution
\begin{equation}\label{eq:h-sol}
    h(t)=(1+t)^{1/2}\cos\bigl(\omega\log(1+t)\bigr), \quad \text{where} \quad \omega:=\sqrt{\mu^2-\frac14}>0.
\end{equation} 
The oscillations are therefore periodic in the logarithmic variable
$\log(1+t)$ rather than in $t$ itself. Since this function changes sign,
it cannot serve as a global positive test function. We instead localize
near the time scales at which the cosine is positive and bounded away from
zero.
For $j\in\mathbb N$, set
\begin{equation}\label{eq:Rj}
R_j:=e^{2\pi j/\omega}-1 \quad \text{so that} \quad 
\omega\log(1+R_j)=2\pi j.    
\end{equation}
Then, $R_j$ corresponds to a maximum of the oscillatory factor. We measure the
logarithmic distance from this maximum by
\begin{equation}\label{eq:sigma_j}
\sigma_j(t):=\omega\log\frac{1+t}{1+R_j}
\end{equation}
and thus
\begin{equation}\label{eq:max-oscil}
    \cos\bigl(\omega\log(1+t)\bigr)
=
\cos\bigl(2\pi j+\sigma_j(t)\bigr)
=
\cos(\sigma_j(t)).
\end{equation}
In particular, on the set where $|\sigma_j(t)|<\pi/4$, this factor is
positive and uniformly bounded from below.

\medskip

Fix $\chi\in C_c^{\infty}(\R)$ such that
\[
0\le \chi\le1,\quad
\chi(s)=1 \ \text{for } |s|\le \frac{\pi}{8},\quad
\chi(s)=0 \ \text{for } |s|\ge \frac{\pi}{4}, \quad \text{and} \quad
\chi(s)>0 \ \text{for } |s|<\frac{\pi}{4}.
\] Moreover,
choose and fix an integer $k\ge 2p'$, and define
\begin{equation}\label{eq:eta_j}
\eta_j(t):=\chi(\sigma_j(t))^k.    
\end{equation}
Furthermore, we also introduce the logarithmic shell
\begin{equation}\label{eq:S_j}
S_j:=\{t>0:\ |\sigma_j(t)|<\pi/4\}
\end{equation}
and its inner core
\begin{equation}\label{eq:I_j}
I_j:=\{t>0:\ |\sigma_j(t)|\le \pi/8\}.
\end{equation}
Observe $\eta_j\equiv1$ on $I_j$, $\eta_j>0$ on $S_j$, and
$\supp\eta_j=\overline{S_j}$.

\medskip

The next lemma collects the geometric properties of these logarithmic shells.

\begin{lemma}\label{lem:shell-geometry}
There exist constants
$0<c_S<c_I<C_I<C_S$,
depending only on $\omega$, such that
\begin{equation}\label{eq:shell-placement}
 I_j\subset[c_IR_j,C_IR_j]
 \quad \text{and} \quad 
 S_j\subset[c_SR_j,C_SR_j],
\end{equation} for all $j\in \N$, where $R_j, S_j, I_j$ are the quantities defined in \eqref{eq:Rj},\eqref{eq:S_j} and \eqref{eq:I_j}.
Consequently, for all $j\in \N$,
\[
|I_j|\asymp R_j
\quad \text{and} \quad
|S_j|\asymp R_j,
\]
with implicit constants depending only on $\omega$. Moreover, 
\begin{equation}\label{eq:A-shell}
1+A(t)\asymp 1+A(R_j),    
\end{equation} for every
$t\in S_j$ and every $j\in \N$,
with implicit constants depending only on $\ell$ and $\omega$.
\end{lemma}

\begin{proof}
If $t\in I_j$ and $\alpha_I:=\pi/(8\omega)$, then
\[
e^{-\alpha_I}(1+R_j)-1
< t < e^{\alpha_I}(1+R_j)-1.
\]
Since $j\in\N$, we have
$R_j\ge R_1=e^{2\pi/\omega}-1>0$;
thus,
\[
1+R_j
=
R_j\left(1+\frac1{R_j}\right)
\le
\left(1+\frac1{R_1}\right)R_j.
\]
Hence,
\[
t
\le
e^{\alpha_I}(1+R_j)
\le
e^{\alpha_I}\left(1+\frac1{R_1}\right)R_j = C_I R_j.
\]

For the lower bound, we write
\[
e^{-\alpha_I}(1+R_j)-1
=
e^{-\alpha_I}R_j-\bigl(1-e^{-\alpha_I}\bigr).
\]
Since $R_j\ge R_1$, we have
\[
1-e^{-\alpha_I}
\le
\frac{1-e^{-\alpha_I}}{R_1}R_j.
\]
Consequently,
\[
e^{-\alpha_I}(1+R_j)-1
\ge
\left(
e^{-\alpha_I}
-
\frac{1-e^{-\alpha_I}}{R_1}
\right)R_j = c_I R_j,
\]
where
\[
c_I := e^{-\alpha_I}
-
\frac{1-e^{-\alpha_I}}{R_1}
=
\frac{e^{-\alpha_I}(1+R_1)-1}{R_1}
=
\frac{e^{2\pi/\omega-\pi/(8\omega)}-1}{R_1}
>0.
\]
Therefore, for all $j \in \N$,
\[
I_j\subset[c_I R_j,C_I R_j].
\]

\smallskip

The proof for $S_j$ is identical, replacing $\alpha_I$ by
$\alpha_S := \pi/(4\omega)$. Hence, there exist positive constants $c_S,C_S$, depending
only on $\omega$, such that, for all $j \in \N$,
\[
S_j\subset[c_S R_j,C_S R_j].
\]

\smallskip

\noindent Since $\alpha_S>\alpha_I$, the shell $S_j$ is larger than $I_j$.
Thus, after possibly decreasing $c_S$ and increasing $C_S$, we may
choose the constants so that
$0<c_S<c_I<C_I<C_S$.
This proves \eqref{eq:shell-placement}.

\smallskip

We now prove the estimates for the lengths. We can describe
$I_j$ explicitly as
\[
I_j
=
\left(
e^{-\alpha_I}(1+R_j)-1,
e^{\alpha_I}(1+R_j)-1
\right);
\]
thus, 
$|I_j|
=
\left(e^{\alpha_I}-e^{-\alpha_I}\right)(1+R_j)$.
Since $j\in\N$ and $R_j\ge R_1>0$, we have
\[
R_j\le 1+R_j
\le
\left(1+\frac1{R_1}\right)R_j.
\]
It follows that, for all $j\in\N$,
\[
|I_j|\asymp R_j,
\]
with implicit constants depending only on $\omega$. The same argument, applied to $S_j$, gives, for all $j\in\N$,
\[
|S_j|
=
\left(e^{\alpha_S}-e^{-\alpha_S}\right)(1+R_j)
\asymp R_j.
\]

\smallskip

It remains to prove \eqref{eq:A-shell}. For $t\in S_j$,
\[
e^{-\alpha_S}
<
\frac{1+t}{1+R_j}
<
e^{\alpha_S},
\]
so that
$1+t\asymp 1+R_j$,
with constants depending only on $\omega$. Since $-1<\ell<0$,
\[
1+A(t)=\frac{(1+t)^{\ell+1}+\ell}{\ell+1},
\]
and $(1+s)^{\ell+1}\ge 1$, for all $s\ge0$, then
\[(1+\ell)(1+s)^{\ell+1} =(1+s)^{\ell+1}+\ell(1+s)^{\ell+1} < (1+s)^{\ell+1} +\ell < (1+s)^{\ell+1};
\]
thus,
\[(1+s)^{\ell+1} \le 1+A(s) < \frac{1}{(1+\ell)}(1+s)^{\ell+1}.\]
This means that $1+A(s)\asymp
(1+s)^{\ell+1}$, with constants depending only on $\ell$. Consequently,
\[
1+A(t)
\asymp
(1+t)^{\ell+1}
\asymp
(1+R_j)^{\ell+1}
\asymp
1+A(R_j).
\]
The implicit constants depend only on $\ell$ and $\omega$. This proves
\eqref{eq:A-shell}.
\end{proof}

The pointwise estimate \eqref{eq:H-pointwise} and \Cref{lem:shell-geometry} immediately yield a lower bound on shell averages.

\begin{proposition}[Lower bound on shells]\label{prop:shell-lower-weighted}
Let $u$ be a weak solution of
\eqref{eq:main2} on $[0,T)$, and let $t_1\ge1$ be given by \Cref{prop:H-pointwise}. Then, there exists a constant $C>0$,
independent of $\varepsilon$, $j$, and $T$, such that, for every
$j\in\N$ with $I_j\subset [t_1,T)$,
\begin{equation}\label{eq:shell-lower-weighted}
 \int_{I_j}(1+t)^{-\alpha}H(t)\,\dd t
 \ge
 C\varepsilon^p
 R_j^{1-\alpha-\ell p/2}
 (1+A(R_j))^{-\frac{n-1}{2}(p-2)}.
\end{equation}
\end{proposition}

\begin{proof}
Let $j\in\N$ be such that $I_j\subset [t_1,T)$. By \Cref{prop:H-pointwise}, for a.e. $t\in [t_1,T)$, we have
\[
H(t)
\ge
C\varepsilon^p
(1+t)^{-\frac{\ell p}{2}}
(1+A(t))^{-\frac{n-1}{2}(p-2)}.
\] In particular,
since $I_j\subset[t_1,T)$, this inequality holds for a.e. $t\in I_j$.
Multiplying by $(1+t)^{-\alpha}$ and integrating over $I_j$, we obtain
\[
\int_{I_j}(1+t)^{-\alpha}H(t)\dd t
\ge
C\varepsilon^p
\int_{I_j}
(1+t)^{-\alpha-\frac{\ell p}{2}}
(1+A(t))^{-\frac{n-1}{2}(p-2)}
\dd t.
\]
Recall
from \Cref{lem:shell-geometry} that, for every $j\in\N$ and every $t \in I_j$, 
\[
1+t\asymp R_j,
\quad
1+A(t)\asymp 1+A(R_j),
\quad \text{and} \quad
|I_j|\asymp R_j,
\] with
the implicit constants independent of $j$. Hence, for every
$t\in I_j$,
\[
(1+t)^{-\alpha-\frac{\ell p}{2}}
\ge
C R_j^{-\alpha-\frac{\ell p}{2}},
\]
and
\[
(1+A(t))^{-\frac{n-1}{2}(p-2)}
\ge
C
(1+A(R_j))^{-\frac{n-1}{2}(p-2)}.
\]
Therefore,
\[
\begin{split}
\int_{I_j}(1+t)^{-\alpha}H(t)\,\dd t
&\ge
C\varepsilon^p
R_j^{-\alpha-\frac{\ell p}{2}}
(1+A(R_j))^{-\frac{n-1}{2}(p-2)}
|I_j| \\
&\ge
C\varepsilon^p
R_j^{1-\alpha-\frac{\ell p}{2}}
(1+A(R_j))^{-\frac{n-1}{2}(p-2)}.\qedhere
\end{split} 
\]
\end{proof}

The preceding proposition ensures a lower bound for nonlinear term on each logarithmic shell. In order to obtain the opposite estimate, we first need uniform control of the derivatives of the logarithmic cutoff $\eta_j$.

\begin{lemma}\label{lem:eta-weighted}
Let $\eta_j$ the functions defined in \eqref{eq:eta_j}. The quantities
\[
|\eta_j'|^{p'}\eta_j^{-1/(p-1)}
\quad \text{and} \quad 
|\eta_j''|^{p'}\eta_j^{-1/(p-1)},
\]
initially defined on $S_j=\{\eta_j>0\}$, extend continuously by zero to all of $[0,\infty)$.
Moreover, there exists $C>0$, independent of $j$, such that
\begin{equation}\label{eq:eta1}
|\eta_j'(t)|^{p'}\eta_j(t)^{-1/(p-1)}
\le
C(1+t)^{-p'}\one_{S_j}(t),
\end{equation}
and
\begin{equation}\label{eq:eta2}
|\eta_j''(t)|^{p'}\eta_j(t)^{-1/(p-1)}
\le
C(1+t)^{-2p'}\one_{S_j}(t).
\end{equation}
\end{lemma}

\begin{proof}
We have
\[
\eta_j'(t)
=
k \, \chi(\sigma_j(t))^{k-1}\chi'(\sigma_j(t)) \, \sigma_j'(t).
\]
Hence,
\[
|\eta_j'(t)|^{p'}\eta_j(t)^{-1/(p-1)}
=
k^{p'}|\chi'(\sigma_j(t))|^{p'}|\sigma_j'(t)|^{p'}
\chi(\sigma_j(t))^{(k-1)p'-k/(p-1)}.
\]
Now, we note
\[
(k-1)p'-\frac{k}{p-1}
=
(k-1)p'-\frac{kp'}{p}
=
k-p'.
\]
Because $k\ge2p'$, this exponent is nonnegative. Since $\sigma_j'(t)=\frac{\omega}{1+t}$
and $\chi$ and $\chi'$ are both bounded, 
we obtain \eqref{eq:eta1} on $S_j$.

Similarly, differentiating once more,
\[
\begin{split}
    \eta_j''(t)
= \ & 
k(k-1)\chi(\sigma_j(t))^{k-2}\chi'(\sigma_j(t))^2\sigma_j'(t)^2
+
k\chi(\sigma_j(t))^{k-1}\chi''(\sigma_j(t))\sigma_j'(t)^2 \\
& +
k\chi(\sigma_j(t))^{k-1}\chi'(\sigma_j(t))\sigma_j''(t).
\end{split}
\]
Thus,
\[
|\eta_j''(t)|
\le
C(1+t)^{-2}\chi(\sigma_j(t))^{k-2}\one_{S_j}(t),
\]
where we use that $\chi,\chi',\chi''$ are bounded and $\sigma_j''(t)=-\frac{\omega}{(1+t)^2}$.
Hence,
\[
|\eta_j''(t)|^{p'}\eta_j(t)^{-1/(p-1)}
\le
C(1+t)^{-2p'}
\chi(\sigma_j(t))^{(k-2)p'-k/(p-1)}.
\]
Now, the exponent
\[
(k-2)p'-\frac{k}{p-1}
=
k-2p'
\] is nonnegative since $k\ge2p'$. Then, we obtain \eqref{eq:eta2} on $S_j$.

Finally, since $\chi\in C_c^{\infty}(\R)$, all derivatives of $\chi$ vanish at the boundary points $\pm\pi/4$ of its support. Therefore, the two expressions above tend to $0$ as $t\to\partial S_j$, and then extend continuously by zero to all of $[0,\infty)$.
\end{proof}

Next, we record two simple consequences of our choice of the logarithmic
shells. Recall that $h$ defined in \eqref{eq:h-sol} solves
\eqref{eq:h-ode-ce} and satisfies \eqref{eq:max-oscil}. Since, for
$t\in S_j$, we have $|\sigma_j(t)|<\pi/4$, we are precisely on a region
where the oscillatory factor has fixed positive sign. More explicitly,
\[
\frac{\sqrt{2}}{2}\le
\cos(\sigma_j(t)) = \cos\bigl(\omega\log(1+t)\bigr) \le 1.
\]
Combining this with the shell localization $1+t\asymp R_j$, obtained in
\Cref{lem:shell-geometry}, we obtain
\begin{equation}\label{eq:h-comparable}
h(t)\asymp R_j^{1/2},
\quad\text{for all }t\in S_j,
\end{equation}
with constants independent of $j$. Moreover, differentiating $h$ gives
\begin{equation}\label{eq:hprime}
|h'(t)|\le C(1+t)^{-1/2},
\quad\text{for all }t\ge0.
\end{equation}

These estimates show that, on each shell, the oscillatory
solution $h$ behaves like a positive constant multiple of $R_j^{1/2}$,
while the derivatives of the cutoff produce only explicit powers of
$(1+t)^{-1}$. This is the mechanism that allows the localized test
function argument to yield the following upper bound.

\begin{proposition}[Upper bound on shells]\label{prop:shell-upper-weighted}
Let $u$ be a finite-propagation weak solution of
\eqref{eq:main2} on $[0,T)$, as in \Cref{def:weak}. Then, there exists a constant
$C>0$, independent of $\varepsilon$, $j$, and $T$, such that, for
every $j\in\N$ satisfying $\overline{S_j}\subset(0,T)$, one has
\begin{equation}\label{eq:shell-upper-weighted}
 \int_{I_j}(1+t)^{-\alpha}H(t)\,\dd t
 \le
 C R_j^{1-2p'+\alpha/(p-1)}
 (1+A(R_j))^n .
\end{equation}
\end{proposition}

\begin{proof}
Fix $j\in\N$ such that $\overline{S_j}\subset(0,T)$ and define
\[
\Psi_j(t):=h(t)\,\eta_j(t).
\]
Then, since $h>0$ on $S_j$, we have
\[
\Psi_j\in C_c^\infty((0,T)),
\quad
\Psi_j\ge0,
\quad
\Psi_j>0 \ \text{on } S_j,
\quad \text{and }
\supp\Psi_j\subset\overline{S_j}.
\]
Moreover, since $h$ solves \eqref{eq:h-ode-ce},
$\Psi_j$ solves
\begin{equation}\label{eq:Psi-error-weighted}
\Psi_j''(t)+\frac{\mu^2}{(1+t)^2}\Psi_j(t)
=
2h'(t)\eta_j'(t)+h(t)\eta_j''(t).
\end{equation}
First, we claim the following identity:
\begin{equation}\label{eq:test-id-weighted}
\int_0^{+\infty}
(1+t)^{-\alpha}\Psi_j(t)H(t)\,\dd t
=
\int_0^{+\infty}
\left(
\Psi_j''(t)+\frac{\mu^2}{(1+t)^2}\Psi_j(t)
\right)
\left(\int_{\R^n}u(x,t)\,\dd x\right)\dd t .
\end{equation}
Here, the left-hand side is the nonlinear term in the weak formulation of
\eqref{eq:main2}.

To prove \eqref{eq:test-id-weighted}, let $\zeta$ be a cutoff function as in \ref{def:smooth_cutoff} and let $c_S,C_S>0$ as in \eqref{eq:shell-placement}. 
Since $\supp\Psi_j\subset\overline{S_j}$, the finite-propagation property
\eqref{eq:finite-speed-weak} implies
\[
\supp u(t,\cdot)
\subset B_{R_0+A(t)}
\subset B_{R_0+A(C_SR_j)},
\quad\text{for every }t\in\supp\Psi_j.
\]
Choose $R_j^\ast>2\bigl(R_0+A(C_SR_j)\bigr)$. Then, for every $t\in\supp\Psi_j$, we have $\supp u(t,\cdot)\subset B_{R_j^\ast/2}$.
Hence, on $\supp u(t,\cdot)$, we have
\begin{equation}\label{eq:zeta_Rjast}
\zeta_{R_j^\ast}\equiv1,
\quad
\nabla\zeta_{R_j^\ast}\equiv0,
\quad \text{and } \
\Delta\zeta_{R_j^\ast}\equiv0.    
\end{equation}
Now take
\[
\phi_j(x,t):=\zeta_{R_j^\ast}(x)\Psi_j(t)
\]
as a test function in the weak formulation of \eqref{eq:main2}.
Since $\Psi_j\in C_c^\infty((0,T))$, the terms with $\phi_j(x,0)$ and $\phi_{j,t}(x,0)$ vanish. Moreover,
\[
\phi_{j,tt}(x,t)=\zeta_{R_j^\ast}(x)\Psi_j''(t),
\qquad
\Delta\phi_j(x,t)=\Psi_j(t)\Delta\zeta_{R_j^\ast}(x).
\]
Thus, the weak formulation gives
\begin{align*}
\int_0^{+\infty}\!\!\!\int_{\R^n}
(1+t)^{-\alpha}|u(x,t)|^p
\zeta_{R_j^\ast}(x)\Psi_j(t) \dd x \dd t
&=
\int_0^{+\infty}\!\!\!\int_{\R^n}
u(x,t)\zeta_{R_j^\ast}(x)
\!\left(
\Psi_j''(t)+\frac{\mu^2}{(1+t)^2}\Psi_j(t)
\right) \!\dd x \dd t
\\
&\quad
-
\int_0^{+\infty}\!\!\!\int_{\R^n}
u(x,t)(1+t)^{2\ell}\Psi_j(t)
\Delta\zeta_{R_j^\ast}(x) \dd x \dd t.
\end{align*}
By our choice of $R_j^\ast$ and \eqref{eq:zeta_Rjast}, we obtain \eqref{eq:test-id-weighted}, as claimed.

\smallskip

Next, in order to simplify some calculations, let us write
\[
L_j:=
\int_0^{+\infty}
(1+t)^{-\alpha}\Psi_j(t)H(t)\dd t.
\]
By \eqref{eq:Psi-error-weighted} and \eqref{eq:test-id-weighted}, we have
\begin{equation}\label{eq:test-id-weighted-2}
L_j
\le
\int_{S_j}
\left|2h'(t)\eta_j'(t)+h(t)\eta_j''(t)\right| \,
\int_{\R^n}\left|u(x,t)\right|\dd x\dd t.
\end{equation}
On the other hand, since $\eta_j\equiv1$ on $I_j$ and $h(t)\ge cR_j^{1/2}$, for all $t\in I_j$, we obtain
\begin{equation}\label{eq:L-lower-weighted}
L_j
\ge
\int_{I_j}(1+t)^{-\alpha}h(t)\eta_j(t)H(t)\,\dd t
\ge
cR_j^{1/2}
\int_{I_j}(1+t)^{-\alpha}H(t)\,\dd t.
\end{equation}
Now, by Hölder's inequality in the space variable and by the
finite-propagation property, there exists $C>0$ such that
\[
\int_{\R^n}\left|u(x,t)\right|\,\dd x
\le
|\supp u(t,\cdot)|^{1/p'}H(t)^{1/p} \le (R_0+A(t))^{n/p'}H(t)^{1/p} \le C(1+A(t))^{n/p'}H(t)^{1/p}.
\]
Substituting this into
\eqref{eq:test-id-weighted-2}, we find
\[
L_j
\le
C\int_{S_j}
\left|2h'(t)\eta_j'(t)+h(t)\eta_j''(t)\right|
(1+A(t))^{n/p'}H(t)^{1/p}\dd t.
\]
Since $\Psi_j>0$ on $S_j$, we can conveniently write
\[
H(t)^{1/p}
=
(1+t)^{\alpha/p}\Psi_j(t)^{-1/p}
\left[(1+t)^{-\alpha}\Psi_j(t)H(t)\right]^{1/p}
\] so that
\[
L_j
\le
C\!\int_{S_j}
\left|2h'(t)\eta_j'(t)+h(t)\eta_j''(t)\right|
(1+A(t))^{n/p'}
(1+t)^{\alpha/p}
\Psi_j(t)^{-1/p}
\left[(1+t)^{-\alpha}\Psi_j(t)H(t)\right]^{1/p}\!\dd t.
\]
Applying Hölder's inequality in $t$ with exponents $p$ and $p'$, we
obtain
\begin{equation}\label{eq:L_j-intermediate}
    L_j\le C M_j^{1/p'}L_j^{1/p},
\end{equation}
where
\begin{equation}\label{eq:M-weighted-def}
M_j
=
\int_{S_j}
\left|2h'(t)\eta_j'(t)+h(t)\eta_j''(t)\right|^{p'}
(1+A(t))^n
(1+t)^{\alpha/(p-1)}
\Psi_j(t)^{-1/(p-1)}
\dd t.
\end{equation}

If $L_j=0$, then $H=0$ for a.e. $t\in I_j$, and
\eqref{eq:shell-upper-weighted} is immediate. Then, we assume
$L_j>0$ and it follows from \eqref{eq:L_j-intermediate} that
\[
L_j\le C M_j.
\]
Combining this with \eqref{eq:L-lower-weighted}, we obtain
\begin{equation}\label{eq:Y-via-M-weighted}
\int_{I_j}(1+t)^{-\alpha}H(t)\,\dd t
\le
CR_j^{-1/2}M_j.
\end{equation}

We now estimate $M_j$. From
\eqref{eq:M-weighted-def} and the elementary inequality $(a+b)^{p'}\le C(a^{p'}+b^{p'})$, we have
\[
M_j\le C\bigl(M_j^{(1)}+M_j^{(2)}\bigr),
\]
where
\begin{align*}
M_j^{(1)}
:=
\int_{S_j}
&(1+t)^{\alpha/(p-1)}
|h'(t)|^{p'}|\eta_j'(t)|^{p'}
\Psi_j(t)^{-1/(p-1)}
(1+A(t))^n\,\dd t,
\end{align*}
and
\begin{align*}
M_j^{(2)}
:=
\int_{S_j}
&(1+t)^{\alpha/(p-1)}
|h(t)|^{p'}|\eta_j''(t)|^{p'}
\Psi_j(t)^{-1/(p-1)}
(1+A(t))^n\,\dd t.
\end{align*}
For the first term, we use the definition $\Psi_j=h\eta_j$, the comparability \eqref{eq:h-comparable},
the estimate \eqref{eq:hprime},
and the weighted cutoff estimate \eqref{eq:eta1}, to obtain
\begin{align*}
M_j^{(1)}
&\le
C R_j^{-1/(2(p-1))}
\int_{S_j}
(1+t)^{\alpha/(p-1)}
(1+t)^{-p'/2}
(1+t)^{-p'}
(1+A(t))^n\dd t
\\
&=
C R_j^{-1/(2(p-1))}
\int_{S_j}
(1+t)^{\alpha/(p-1)-3p'/2}
(1+A(t))^n\dd t.
\end{align*}
On $S_j$, the shell geometry from \Cref{lem:shell-geometry} gives
$1+t\asymp R_j$, $|S_j|\asymp R_j$, and
$1+A(t)\asymp 1+A(R_j)$,
with constants independent of $j$. Hence,
\begin{equation}\label{eq:M1-weighted}
M_j^{(1)}
\le
C R_j^{-1/(2(p-1))}
R_j^{1+\alpha/(p-1)-3p'/2}
(1+A(R_j))^n.
\end{equation}

For the second term, we use instead the weighted cutoff estimate \eqref{eq:eta2}
to infer
\begin{align*}
M_j^{(2)}
&\le
C R_j^{p'/2-1/(2(p-1))}
\int_{S_j}
(1+t)^{\alpha/(p-1)}
(1+t)^{-2p'}
(1+A(t))^n\dd t.
\end{align*}
Using again the shell geometry,
\begin{equation}\label{eq:M2-weighted-pre}
M_j^{(2)}
\le
C R_j^{p'/2-1/(2(p-1))}
R_j^{1+\alpha/(p-1)-2p'}
(1+A(R_j))^n.
\end{equation}
Since
\[
\frac{p'}2-\frac{1}{2(p-1)}=\frac12,
\]
we obtain
\begin{equation}\label{eq:M2-weighted}
M_j^{(2)}
\le
C R_j^{1/2}
R_j^{1+\alpha/(p-1)-2p'}
(1+A(R_j))^n.
\end{equation}
Observe  that the powers in \eqref{eq:M1-weighted} and \eqref{eq:M2-weighted} are the
same and so
\begin{equation}\label{eq:M-weighted-final}
M_j
\le
C
R_j^{-1/(2(p-1))}
R_j^{1+\alpha/(p-1)-3p'/2}
(1+A(R_j))^n.
\end{equation}

Substituting \eqref{eq:M-weighted-final} into
\eqref{eq:Y-via-M-weighted}, we arrive at
\[
\int_{I_j}(1+t)^{-\alpha}H(t)\,\dd t
\le
C
R_j^{-1/2}
R_j^{-1/(2(p-1))}
R_j^{1+\alpha/(p-1)-3p'/2}
(1+A(R_j))^n.
\] Since the combined exponent of $R_j$ is equal to
\[
-\frac12-\frac{1}{2(p-1)} + 1+\frac{\alpha}{p-1}
-\frac{3p'}2
=
1-2p'+\frac{\alpha}{p-1},
\] we conclude \eqref{eq:shell-upper-weighted}.
\end{proof}

\section{Blow up and Lifespan}\label{sec:proof-main}

We are now ready to combine the upper and lower bounds on shells and to prove \Cref{thm:lifespan-weighted}. 

\begin{proof}[Proof of \Cref{thm:lifespan-weighted}]
Set
\begin{equation}\label{eq:Theta-gamma-alpha}
\Theta_\alpha
:=
1-2p'+\frac{\alpha}{p-1}
-
\left(1-\alpha-\frac{\ell p}{2}\right)
+
(1+\ell)
\left(
n+\frac{n-1}{2}(p-2)
\right)
=
\frac{(1+\ell)\gamma_\alpha(n,\ell;p)}{2(p-1)}.
\end{equation}

\smallskip

\noindent
Since $\gamma_\alpha(n,\ell;p)<0$, we have $\Theta_\alpha<0$.
Let $t_1\ge1$ be given by \Cref{prop:H-pointwise} and $0<c_S<c_I<C_I<C_S$ given by \Cref{lem:shell-geometry}. 
Choose $j_\ast\ge1$ such that
\begin{equation}\label{eq:jstar-lower}
c_I R_{j_\ast}\ge t_1.
\end{equation}
This is possible since $R_j\to\infty$. We split the argument into two
cases.

\smallskip

Assume first that $T_\eps\le C_S R_{j_\ast}$. Then the desired lifespan estimate is immediate: if
\[
\sigma_\alpha:=\frac{p}{\Theta_\alpha}=
\frac{2p(p-1)}
{(1+\ell)\gamma_\alpha(n,\ell;p)}<0;
\]
then, for $0<\eps\le1$,
\[
T_\eps
\le
C_S R_{j_\ast}
\le
C_S R_{j_\ast}\eps^{\sigma_\alpha} \le C\eps^{\sigma_\alpha}.
\]

\smallskip

Let us then assume, on the other hand, that $T_\eps>C_S R_{j_\ast}$. 
In this case, at least one $j$, namely $j=j_*$, satisfies $\overline{S_{j}}\subset(0,T_\eps)$. Moreover, by \eqref{eq:jstar-lower} and $I_j\subset S_j$, we have
$I_{j}\subset[c_I R_{j},C_I R_{j}]
\subset[t_1,T_\eps)$. 
Hence, both \Cref{prop:shell-lower-weighted} and
\Cref{prop:shell-upper-weighted} apply and we have
\[
c\eps^p
R_j^{1-\alpha-\ell p/2}
(1+A(R_j))^{-\frac{n-1}{2}(p-2)}
\le
\int_{I_j}(1+t)^{-\alpha}H(t)\,\dd t
\le
C
R_j^{1-2p'+\alpha/(p-1)}
(1+A(R_j))^n .
\]
Thus,
\[
\eps^p
\le
C
R_j^{
1-2p'+\alpha/(p-1)
-\left(1-\alpha-\frac{\ell p}{2}\right)
}
(1+A(R_j))^{n+\frac{n-1}{2}(p-2)}.
\]
Since $1+A(R_j)\asymp R_j^{1+\ell}$,
we obtain, for all $j\in\N$ such that $\overline{S_{j}}\subset(0,T_\eps)$,
\begin{equation}\label{eq:eps-Theta}
\eps^p\le C R_j^{\Theta_\alpha}.
\end{equation}

Now, we argue that, since $\Theta_\alpha<0$, \eqref{eq:eps-Theta} implies that $T_\eps< +\infty$. Indeed, if $T_\eps=+\infty$, then
$\overline{S_j}\subset(0,T_\eps)$, for every $j\ge j_\ast$, and hence, for every $j\ge j_\ast$,
\[
\eps^p\le C R_j^{\Theta_\alpha}.
\]
Letting $j\to\infty$, the right-hand side tends to $0$ and we have a contradiction.
Therefore, $T_\eps<+\infty$ and no global weak solution can exist.

For the lifespan, choose $j_\eps\ge j_\ast$ such that
\begin{equation}\label{eq:choice-jeps-simple}
C_S R_{j_\eps}<T_\eps\le C_S R_{j_\eps+1}.
\end{equation}
Then $\overline{S_{j_\eps}}\subset(0,T_\eps)$ and \eqref{eq:eps-Theta} and $\Theta_\alpha<0$ imply that
\[
R_{j_\eps}\le C\eps^{p/\Theta_\alpha}
=
C\eps^{\sigma_\alpha}.
\]
Finally, since there exists a constant $C_\omega$ such that $R_{j+1}\le C_\omega R_j$, for all $j\in\N$ (for example, $C_\omega=e^{2\pi/\omega}+1$), we use \eqref{eq:choice-jeps-simple} to obtain
\[
T_\eps
\le
C_S R_{j_\eps+1}
\le
C R_{j_\eps}
\le
C\eps^{\sigma_\alpha}.
\]
Therefore,
\[
T_\eps
\le
C\eps^{
\frac{2p(p-1)}
{(1+\ell)\gamma_\alpha(n,\ell;p)}
},
\] which is \eqref{eq:lifespan-weighted}.
\end{proof}

\vspace{2em}

\textbf{Data availability:} Data sharing not applicable to this article as no datasets were generated or analysed during
the current study.

\vspace{1em}

\textbf{Conflict of interest:} We do not have any conflict of interest to declare

\newpage

\bibliographystyle{abbrv}
\bibliography{ref.bib}

\newpage

\appendix 
\section{Proof of \Cref{prop:adjoint-profile}} \label{appendixA}

	For fixed $\lambda>|\mu|$, introduce
	\[
	b_0(t):=\lambda(1+t)^\ell+\frac{\ell}{2(1+t)}
	\quad \text{and} \quad
	q:=\mu^2+\frac{\ell(\ell+2)}{4}.
	\]
	Since $-1<\ell<0$ and $\mu^2>\frac14$, we have $b_0 > 0$ and 
	$
	q>0$.
	The choice of $b_0$ is suggested by the formal profile
	$
	(1+t)^{-\ell/2}e^{-\lambda A(t)}
	$,
	since
	\[
	\frac{\dd}{\dd t}\Bigl((1+t)^{-\ell/2}e^{-\lambda A(t)}\Bigr)
	=
	-
	b_0 (t) \, 
	(1+t)^{-\ell/2}e^{-\lambda A(t)},
	\] and we look for a positive solution $m$ of \eqref{eq:m-eq} in the form
	\begin{equation}\label{eq:m-lam-ode}
	    -\frac{m'(t)}{m(t)}=b_0(t)+\eta(t),
	\end{equation}
	where $\eta$ is a suitable correction term.
	
	We will split the proof into 5 steps.

	\subsection{The Riccati equation} 
    
    Set $r=b_0+\eta$ and observe that $m'= - mr$ implies $m''=(r^2-r')m$. Thus \eqref{eq:m-eq} is equivalent to
	\[
	r(t)^2-r'(t)
	=
	\lambda^2(1+t)^{2\ell}-\frac{\mu^2}{(1+t)^2}.
	\]
	A direct computation gives
	\[
	b_0(t)^2-b_0'(t)
	=
	\lambda^2(1+t)^{2\ell}+\frac{\ell(\ell+2)}{4(1+t)^2}
	\]
	and hence, by our choice of $q > 0$, the correction term $\eta$ must solve
	\begin{equation}\label{eq:eta-riccati-proof}
		\eta'(t)-2b_0(t)\eta(t)=\eta(t)^2+\frac{q}{(1+t)^2}.
	\end{equation}
	We impose the condition that $\eta(t)\to0$, as $t\to+ \infty$, which, by an integrating factor argument, leads to the integral equation
	\begin{equation}\label{eq:eta-integral-proof}
		\eta(t)
		=
		-\int_t^{+\infty}
		K(t,s)\biggl(\eta(s)^2+\frac{q}{(1+s)^2}\biggr)\dd s,
	\end{equation}
	where, for $s\ge t$,
	\begin{equation}\label{eq:kernel-proof}
		K(t,s):=
		\exp\,\biggl(-2\int_t^s b_0(\sigma)\dd\sigma\biggr).
	\end{equation}
	Since
	\[
	2\int_t^s b_0(\sigma)\dd\sigma
	=
	2\lambda \big[A(s)-A(t)\big]+\ell\log\frac{1+s}{1+t},
	\]
	we can rewrite the kernel $K$ as
	\begin{equation}\label{eq:kernel-explicit-proof}
		K(t,s)
		=
		e^{-2\lambda\big[A(s)-A(t)\big]}
		\biggl(\frac{1+t}{1+s}\biggr)^\ell.
	\end{equation}

	\subsection{Kernel estimates}
    
    Next,
	we claim that, for all $t\ge0$,
	\begin{equation}\label{eq:kernel-est-1-proof}
		\int_t^{+\infty} K(t,s)\frac{1}{(1+s)^2}\dd s
		\le
		\frac{1}{2 \lambda} \, (1+t)^{-\ell-2}
	\end{equation}
	and
	\begin{equation}\label{eq:kernel-est-2-proof}
		\int_t^{+\infty} K(t,s)\frac{1}{(1+s)^{2\ell+4}}\dd s
		\le
		\frac{1}{2 \lambda} \, (1+t)^{-3\ell-4}.
	\end{equation}
	
	In fact, the expression \eqref{eq:kernel-explicit-proof} for the kernel gives
	\[
	\int_t^{+\infty} K(t,s)\frac{1}{(1+s)^2}\dd s
	=
	(1+t)^\ell
	\int_t^{+\infty} e^{-2\lambda\big[A(s)-A(t)\big]}(1+s)^{-\ell-2}\dd s.
	\]
	Then, since $A'(s)= a(s) = (1+s)^\ell$, the change of variable
	$
	y=A(s)-A(t)$ implies
	\[
	\int_t^{+\infty} K(t,s)\frac{1}{(1+s)^2}\dd s
	=
	(1+t)^\ell
	\int_0^{+\infty} e^{-2\lambda y}(1+s(y))^{-2\ell-2}\dd y,
	\]
	and we conclude \eqref{eq:kernel-est-1-proof} by noting that
	\[
	\big[1+s(y)\big]^{-2(\ell+1)}
	=
	\bigl[(1+t)^{\ell+1}+(\ell+1)y\bigr]^{-2}
	\le
	(1+t)^{-2\ell-2}.
	\]
	
	The proof of \eqref{eq:kernel-est-2-proof} is entirely analogous.
	
	\subsection{Fixed-point argument at infinity}

	Fix $T\ge1$, to be chosen sufficiently large, and let
	\[
	\mathrm{E}:=
	\Bigl\{
	\eta\in C([T,\infty)):\ 
	\norm{\eta}_{\mathrm{E}}:=
	\sup_{t\ge T} \ (1+t)^{\ell+2}\abs{\eta(t)}<\infty
	\Bigr\}.
	\]
	Moreover, for $M>0$ to be chosen later, define
	\[
	\mathcal B:=
	\Bigl\{
	\eta\in \mathrm{E}:\ 
	\eta(t)\le0, \text{ for } t\ge T, \text{ and } 
	\norm{\eta}_{\mathrm{E}}\le M
	\Bigr\}
	\] and
	\[
	(\mathcal T\eta)(t)
	:=
	-\int_t^{+\infty}
	K(t,s)\biggl(\eta(s)^2+\frac{q}{(1+s)^2}\biggr)\dd s.
	\]
	Clearly, $\mathcal T\eta\le0$ for all $\eta\in \mathrm{E}$.
	
	\smallskip
	
	First, we show how to choose $M > 0$ and $T\ge 1$ so that $\mathcal T$ maps $\mathcal B$ into itself. Indeed, if $\eta\in\mathcal B$, then
	\[
	\abs{\eta(s)}\le M(1+s)^{-\ell-2}
	\quad \text{and} \quad
	\eta(s)^2\le M^2(1+s)^{-2\ell-4}.
	\]
	By \eqref{eq:kernel-est-1-proof} and \eqref{eq:kernel-est-2-proof}, we obtain
	\begin{align*}
		\abs{(\mathcal T\eta)(t)}
		&\le
		q\int_t^{+\infty} K(t,s)\frac{1}{(1+s)^2}\dd s
		+
		M^2\int_t^{+\infty} K(t,s)\frac{1}{(1+s)^{2\ell+4}}\dd s  \\[0.4em]
		&\le
		\frac{1}{2 \lambda} \Big( q(1+t)^{-\ell-2}
		+
		M^2(1+t)^{-3\ell-4} \Big).
	\end{align*}
	Multiplying by $(1+t)^{\ell+2}$, the choice $M:= \tfrac{q}{\lambda} > 0$ yields
    \vspace{0.2em}
	\begin{equation*}
		\norm{\mathcal T\eta}_{\mathrm{E}}
		\le
		\frac{q}{2 \lambda}+ \frac{M^2T^{-2\ell-2}}{2 \lambda} = \frac{M}{2} \bigg( 1 + \frac{qT^{-2\ell-2}}{\lambda^{2}} \bigg).
	\end{equation*}
    Next, since $\ell+1>0$, we choose $T\ge1$ large enough so that
\begin{equation}\label{eqn:choiceTM}
    \frac{qT^{-2\ell-2}}{\lambda^{2}}\le \frac12
    \quad\text{or, equivalently,}\quad
    T^{2(\ell+1)}\ge \frac{2q}{\lambda^2}.
\end{equation}
	The choice of the factor $1/2$ is convenient and is also used
in the contraction estimate below. It follows that
	$
	\norm{\mathcal T\eta}_{\mathrm{E}}\le M$ 
	and that $\mathcal T$ maps $\mathcal B$ into itself, as claimed.
	
	\smallskip
	
	Now, 
	if $\eta_1,\eta_2\in\mathcal B$, then the definition of the norm implies 
	\[
	\abs{\eta_1(s)^2-\eta_2(s)^2}
	\le
	(\abs{\eta_1(s)}+\abs{\eta_2(s)})\abs{\eta_1(s)-\eta_2(s)}
	\le
	\dfrac{2q}{\lambda} \norm{\eta_1-\eta_2}_{\mathrm{E}}(1+s)^{-2\ell-4}.
	\]
	Thus, by  \eqref{eq:kernel-est-2-proof}, 
	\begin{align*}
		\abs{(\mathcal T\eta_1)(t)-(\mathcal T\eta_2)(t)}
		&\le
		\dfrac{2q}{\lambda}\norm{\eta_1-\eta_2}_{\mathrm{E}}
		\int_t^{+ \infty} K(t,s)(1+s)^{-2\ell-4}\dd s \\
		&\le
		\frac{q}{\lambda^{2}} \, \norm{\eta_1-\eta_2}_{\mathrm{E}}(1+t)^{-3\ell-4}.
	\end{align*}
	We multiply by $(1+t)^{\ell+2}$ and recall \eqref{eqn:choiceTM} to obtain
	\[
	\norm{\mathcal T\eta_1-\mathcal T\eta_2}_{\mathrm{E}}
	\le
	\frac{qT^{-2\ell-2}}{\lambda^{2}} \, \norm{\eta_1-\eta_2}_{\mathrm{E}} \le
	\frac12 \norm{\eta_1-\eta_2}_{\mathrm{E}}
	\] and
	it follows that $\mathcal T$ is a contraction on $\mathcal B$. Since $\mathrm E$ is a Banach space and $\mathcal B$ is a closed subset of $\mathrm E$, the contraction mapping theorem applies. Hence, there exists a unique fixed point $\eta\in\mathcal B$ satisfying \eqref{eq:eta-integral-proof} on $[T,+\infty)$.
	
	\subsection{Reconstruction of $m_\lambda$}
    
	The integrand in \eqref{eq:eta-integral-proof} is continuous in $(t,s)$ for $s\ge t$, and the estimates above provide an integrable dominating function on every compact $t$-interval in $[T,+\infty)$. Then, differentiation under the integral sign is justified. Since
	\[
	K(t,t)=1
	\quad \text{and} \quad
	\p_t K(t,s)=2b_0(t)K(t,s),
	\]
	we obtain
	\[
	\eta'(t)
	=
	\eta(t)^2+\frac{q}{(1+t)^2}
	-
	2b_0(t)\int_t^{+ \infty} K(t,s)\biggl(\eta(s)^2+\frac{q}{(1+s)^2}\biggr)\dd s.
	\]
	Using \eqref{eq:eta-integral-proof}, this becomes
	\[
	\eta'(t)
	=
	\eta(t)^2+\frac{q}{(1+t)^2}+2b_0(t)\eta(t).
	\]
	It follows that $\eta\in C^1([T,\infty))$ and, thus, it solves \eqref{eq:eta-riccati-proof} classically.
	
	\smallskip
	
	Now, for $t\ge T$, define
	\begin{equation}\label{eq:m-large-proof}
		m_\lambda(t)
		:=
		(1+t)^{-\ell/2}e^{-\lambda A(t)}
		\exp\,\biggl(-\int_T^t \eta(s)\dd s\biggr).
	\end{equation} Observe that $m_\lambda$ solves \eqref{eq:m-lam-ode} and $\eta$ solves \eqref{eq:eta-riccati-proof}; then, a direct computation shows that $m_\lambda$ satisfies \eqref{eq:m-eq} on $[T,\infty)$.
	Since $\ell+2>1$ and, by construction, $\eta\le0$ in $[T, + \infty)$ and
	\[
	\abs{\eta(s)}\le \dfrac{q}{\lambda} \, (1+s)^{-\ell-2},
	\]
	we have that $- \eta = \abs{\eta}$ is integrable in $[T, + \infty)$; more precisely, 
		\[
		0 \le - \int_T^{+ \infty} \eta(s) \dd s \le \dfrac{q}{\lambda} \int_T^{+ \infty} (1+s)^{-\ell-2} \dd s = \frac{q}{\lambda (\ell+1)}(1+T)^{-\ell-1};
		\] then, by \eqref{eqn:choiceTM},
		\[
		\frac{q}{\lambda (\ell+1)}(1+T)^{-\ell-1} \le \frac{q}{\lambda (\ell+1)}T^{-\ell-1} \le \frac{\sqrt q}{\sqrt2(\ell+1)}
		\] Thus, for all $t\ge T$,
		\[
		1
		\le
		\exp\Bigl(-\int_T^t \eta(s)\,\dd s\Bigr)
		\le
		\exp\left(\frac{\sqrt q}{\sqrt2(\ell+1)}\right).
		\] Since $a(t)^{-1/2}=(1+t)^{-\ell/2}$, this proves \eqref{eq:m-bounds} on $[T,\infty)$.
	
	Moreover, from the definition \eqref{eq:m-large-proof}, we have
	\[
	-\frac{m_\lambda'(t)}{m_\lambda(t)}
	=
	b_0(t)+\eta(t)
	=
	\lambda(1+t)^\ell+\frac{\ell}{2(1+t)}+\eta(t).
	\]
	Since both $(1+t)^{-1}=\mathrm{o}((1+t)^\ell)$ and $\eta(t)=\mathrm{O}((1+t)^{-\ell-2})=\mathrm{o}((1+t)^\ell)$ as $t\to\infty$, there exists $T_1\ge T$ such that, for all $t\ge T_1$,
	\[
	\frac{\lambda}{2}(1+t)^\ell
	\le
	-\frac{m_\lambda'(t)}{m_\lambda(t)}
	\le
	\frac{3\lambda}{2}(1+t)^\ell.
	\] 
	In particular,
	$
	m_\lambda'(t)<0$, for all $t\ge T_1$.
	Now, multiply by $m_\lambda(t)$ and use \eqref{eq:m-bounds} to obtain \eqref{eq:mprime-bounds} on $[T_1,\infty)$. 
	
	\subsection{Conclusion}

	Solve the linear ODE \eqref{eq:m-eq} backward on $[0,T_1]$ with terminal data
	\[
	m_\lambda(T_1)>0
	\quad \text{and} \quad
	m_\lambda'(T_1)<0.
	\]
	By uniqueness for linear ODEs, this defines a function
	$
	m_\lambda\in C^2([0,\infty))$
	which coincides with function the previously constructed on $[T_1,\infty)$ and satisfies \eqref{eq:m-eq} on the whole half-line.
	Set
	\[
	V_\lambda(t):=\lambda^2(1+t)^{2\ell}-\frac{\mu^2}{(1+t)^2}.
	\]
	Since $\lambda>|\mu|$ and $\ell>-1$, we have, for all $t\ge0$,
	\[
	V_\lambda(t)
	=
	\frac{\lambda^2(1+t)^{2(\ell+1)}-\mu^2}{(1+t)^2}
	\ge
	\frac{\lambda^2-\mu^2}{(1+t)^2}
	>0.
	\]
	Consider
	$P(t):=m_\lambda(t)m_\lambda'(t)$.
	By \eqref{eq:m-eq}, for all $t\ge0$,
	\[
	P'(t)=m_\lambda'(t)^2+m_\lambda(t)m_\lambda''(t)
	=
	m_\lambda'(t)^2+V_\lambda(t)m_\lambda(t)^2
	>0.
	\]
	Hence, $P$ is strictly increasing. Since
	$P(T_1)=m_\lambda(T_1)m_\lambda'(T_1)<0$,
	it follows that, for all $t\in[0,T_1]$,
	\[
	P(t)<P(T_1)<0.
	\]
	In particular, $P(t)\neq0$ on $[0,T_1]$. If $m_\lambda(t_0)=0$ for some $t_0\in[0,T_1]$, then $P(t_0)=0$, a contradiction. Hence, $m_\lambda$ does not vanish on $[0,T_1]$. Since $m_\lambda(T_1)>0$, continuity implies
	\[
	m_\lambda(t)>0,
	\quad\text{for all }t\in[0,T_1].
	\]
	Because $P(t)=m_\lambda(t)m_\lambda'(t)<0$ and $m_\lambda(t)>0$, we conclude that
	\[
	m_\lambda'(t)<0,
	\quad\text{for all }t\in[0,T_1].
	\]
	Finally, taking $t_0:=T_1$, all the asserted properties follow.

\end{document}